\documentclass[10 pt, draft]{article}

\title{The Torelli group and congruence subgroups of the mapping class group}
\author{Andrew Putman\footnote{Supported in part by NSF grant DMS-1005318}}

\usepackage[compact]{titlesec}
\usepackage{setspace}
\usepackage{amsmath}
\usepackage{amssymb}
\usepackage{amsthm}
\usepackage{epsfig}
\usepackage[margin=1.7in]{geometry}
\usepackage{amsfonts}
\usepackage{calc}
\usepackage{amscd}
\usepackage[font=small]{caption}
\usepackage{pinlabel}
\usepackage[all,cmtip]{xy}
\usepackage{multirow}

\theoremstyle{plain}
\newtheorem{theorem}{Theorem}[section]

\newtheorem{lemma}[theorem]{Lemma}

\theoremstyle{definition}

\newtheorem*{definition}{Definition}

\theoremstyle{remark}
\newtheorem{exercise}[theorem]{Exercise}
\newtheorem*{remark}{Remark}
\newtheorem*{remarks}{Remarks}
\newtheorem*{example}{Example}

\DeclareMathOperator{\Hom}{Hom}

\DeclareMathOperator{\Ker}{ker}
\DeclareMathOperator{\Image}{Im}

\DeclareMathOperator{\Mod}{Mod}
\newcommand\Torelli{\ensuremath{{\mathcal I}}}

\DeclareMathOperator{\Sp}{Sp}

\DeclareMathOperator{\SL}{SL}

\DeclareMathOperator{\U}{U}

\newcommand\CNosep{\ensuremath{\mathcal{C}^{\text{ns}}}}

\newcommand\Torus{\ensuremath{\mathbb{T}}}

\newcommand\Z{\ensuremath{\mathbb{Z}}}
\newcommand\Q{\ensuremath{\mathbb{Q}}}

\DeclareMathOperator{\HH}{H}
\DeclareMathOperator{\HHH}{\mathcal{H}}

\newcommand\Set[2]{\ensuremath{\{\text{#1 $|$ #2}\}}}
\newcommand\GroupPres[2]{\ensuremath{\langle #1 \text{ $|$ } #2 \rangle}}

\newcommand\Figure[3]{
\begin{figure}[t]
\centering
\centerline{\psfig{file=#2,scale=60}}
\caption{#3}
\label{#1}
\end{figure}}

\newcommand\SpLie{\ensuremath{\mathfrak{sp}}}
\newcommand\SLLie{\ensuremath{\mathfrak{sl}}}
\newcommand\One{\ensuremath{\mathbb{I}}}
\newcommand\Zero{\ensuremath{\mathbb{O}}}
\newcommand\MatTwoTwo[4]{\ensuremath{\left( \begin{smallmatrix} #1 & #2\\ #3 & #4 \end{smallmatrix} \right)}}

\newcommand\ParagraphHeading[1]{\medskip \noindent {\bf #1}}

\begin{document}

\maketitle

\section*{Introduction}

Let $\Sigma_{g,n}$ be a compact oriented genus $g$ surface with $n$ boundary components.
The {\em mapping class group} of $\Sigma_{g,n}$, denoted $\Mod_{g,n}$, is the group
of orientation-preserving diffeomorphisms of $\Sigma_{g,n}$ that restrict to the identity on
$\partial \Sigma_{g,n}$, modulo isotopies that fix $\partial \Sigma_{g,n}$.
The group $\Mod_{g,n}$ plays a fundamental role in many areas of mathematics, ranging
from low-dimensional topology to algebraic geometry.  At least for degrees less than about
$2g/3$, the cohomology of $\Mod_{g,n}$ is well-understood due to the resolution of the
Mumford conjecture by Madsen and Weiss \cite{MadsenWeiss} (together with Harer's unpublished
improved version of his homological stability theorem; see \cite{Boldsen} for an exposition
of this).

However, the cohomology of finite-index subgroups of $\Mod_{g,n}$ remains a mystery.  In
these notes, we will focus on one low-degree calculation.  Consider $n \in \{0,1\}$.
For an integer $p$, the {\em level $p$ congruence subgroup} of $\Mod_{g,n}$, denoted
$\Mod_{g,n}(p)$, is the subgroup of $\Mod_{g,n}$ consisting of mapping classes that
act trivially on $\HH_1(\Sigma_{g,n};\Z/p)$.  Another description of $\Mod_{g,n}(p)$ is
as follows.  The action of $\Mod_{g,n}$ on $\HH_1(\Sigma_{g,n};\Z)$ preserves the algebraic
intersection pairing.  Since $n \leq 1$, this is a nondegenerate alternating form, so we
obtain a representation $\Mod_{g,n} \rightarrow \Sp_{2g}(\Z)$.  Classically this representation
was known to be surjective (see \S \ref{appendix:symplectic}).  Let $\Sp_{2g}(\Z,p)$ be the subgroup of $\Sp_{2g}(\Z)$ consisting
of matrices which equal the identity modulo $p$.  Then $\Mod_{g,n}(p)$ is the pullback
of $\Sp_{2g}(\Z,p)$ to $\Mod_{g,n}$.

These notes will discuss the calculation of $\HH^2(\Mod_{g,n}(p);\Z)$.  One motivation
for this is the study of line bundles on the finite cover of the moduli space of curves
associated to $\Mod_{g,n}(p)$, which is known as the moduli space of curves with level $p$ structures.
The first Chern class of such a line bundle lies in $\HH^2(\Mod_{g,n}(p);\Z)$, and
the determination of $\HH^2(\Mod_{g,n}(p);\Z)$ is the heart of the paper \cite{PutmanPicardGroupLevel},
which gives a complete classification of such line bundles.  However, in these notes
we will ignore this connection to algebraic geometry.  Instead, we will use the
computation of this cohomology group as an excuse to discuss a number of interesting
topics related to the mapping class group.

The universal coefficients exact sequence for $\HH^2(\Mod_{g,n}(p);\Z)$ takes the form
\begin{align*}
0 \longrightarrow \text{Ext}(\HH_1(\Mod_{g,n}(p);\Z),\Z) &\longrightarrow \HH^2(\Mod_{g,n}(p);\Z) \\
&\ \ \ \ \ \ \ \ \longrightarrow \Hom(\HH_2(\Mod_{g,n}(p);\Z),\Z) \longrightarrow 0.
\end{align*}
The third and fourth lecture will be devoted to calculating the kernel and cokernel of this
exact sequence.  They will be proceeded by two lectures on necessary background.  Let us
now give a more detailed description of the four lectures.
\begin{itemize}
\item Lecture 1 will be devoted to the {\em Torelli group}.  Denoted $\Torelli_{g,n}$, this
is the subgroup of $\Mod_{g,n}$ consisting of mapping classes that act trivially on $\HH_1(\Sigma_{g,n};\Z)$.
There are short exact sequences
$$1 \longrightarrow \Torelli_{g,n} \longrightarrow \Mod_{g,n} \longrightarrow \Sp_{2g}(\Z) \longrightarrow 1$$
and
$$1 \longrightarrow \Torelli_{g,n} \longrightarrow \Mod_{g,n}(p) \longrightarrow \Sp_{2g}(\Z,p) \longrightarrow 1,$$
and the structure of $\Mod_{g,n}(p)$ is a sort of mixture of the structure of $\Torelli_{g,n}$ and
$\Sp_{2g}(\Z,p)$.
\item Lecture 2 will be devoted to the {\em Johnson homomorphism}.  Set $\HHH = \HH_1(\Sigma_{g,n};\Z)$.
The Johnson homomorphism is a surjective homomorphism $\tau : \Torelli_{g,1} \longrightarrow \wedge^3 \HHH$.
A deep theorem of Johnson shows that the Johnson homomorphism gives the ``rational part''
of the abelianization of $\Torelli_{g,1}$.  More precisely, 
$\HH_1(\Torelli_{g,1};\Z) \cong W \oplus \wedge^3 \HHH$, where $W$ consists of torsion (in fact, $2$-torsion).
We will also construct a ``mod $p$'' version of the Johnson homomorphism which takes the form
$\tau_p : \Mod_{g,1}(p) \longrightarrow \HHH_p$, where $\HHH_p = \HH_1(\Sigma_{g,n};\Z/p)$.
\item Lecture 3 is devoted to calculating $\HH_1(\Mod_{g,1}(p);\Z)$ for odd $p$.  See
the beginning of that lecture for why we restrict to odd $p$ and do not consider the closed case.  There
are two basic pieces.  The first comes from the mod $p$ Johnson homomorphism and the second comes
from the abelianization of $\Sp_{2g}(\Z,p)$.
\item Lecture 4 is devoted to proving that $\HH_2(\Mod_{g}(p);\Q) \cong \Q$.  Of course, this
implies that
$$\Hom(\HH_2(\Mod_{g,n}(p);\Z),\Z) \cong \Z.$$
The major work here is related to homological stability.
\end{itemize}

\section*{Lecture 1 : The Torelli group}
\addtocounter{section}{1}
\setcounter{theorem}{0}
\setcounter{figure}{0}

The Torelli group was first considered by Nielsen and Magnus in the early 20th century.  However, 
its study only really took off in the late
'70's and early '80's thanks to work of several people, most especially Birman and Johnson.
Johnson's work has proven particularly fundamental and influential, and his survey \cite{JohnsonSurvey}
cannot be recommended enough.

Throughout this lecture, we will fix some $n \in \{0,1\}$.

\ParagraphHeading{The symplectic representation.}
Recall that $\Torelli_{g,n}$ is the kernel of the representation $\Mod_{g,n} \rightarrow \Sp_{2g}(\Z)$
arising from the action of $\Mod_{g,n}$ on $\HH_1(\Sigma_{g,n};\Z)$.  We will need the following fact
about this action.  If $x$ is a simple closed curve on $\Sigma_{g,n}$, then let $T_x$ denote the
right Dehn twist about $x$.  Also, let $i_a(\cdot,\cdot)$ denote the algebraic intersection pairing
on $\HH_1(\Sigma_{g,n};\Z)$.

\begin{lemma}
\label{lemma:transvection}
Let $x$ be a simple closed curve on $x$.  Orient $x$ in an arbitrary way, and let
$[x] \in \HH_1(\Sigma_{g,1};\Z)$ denote the associated homology class.  Then for
$v \in \HH_1(\Sigma_{g,1};\Z)$, we have
$$T_x(v) = v + i_a([x],v) \cdot [x].$$
\end{lemma}

\begin{remark}
The Dehn twist $T_x$ does not depend on an orientation on $x$.  As a sanity check, you should
verify that despite its appearance, the formula in Lemma \ref{lemma:transvection} does not
depend on the orientation of $x$.
\end{remark}

\Figure{figure:generators}{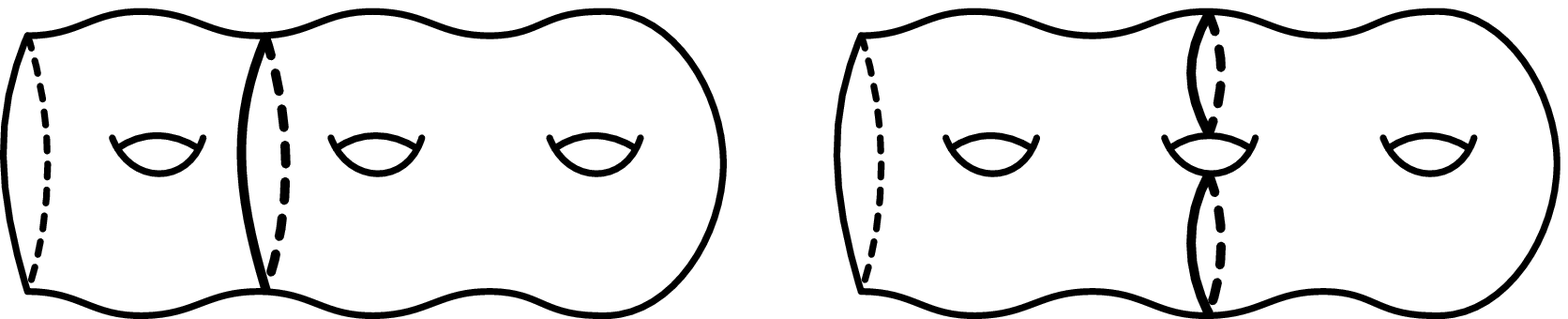}{A separating twist $T_x$ and a
bounding pair map $T_y T_{y'}^{-1}$}

\begin{exercise}
Prove Lemma \ref{lemma:transvection}.
\end{exercise}

\ParagraphHeading{Basic elements of Torelli.}
Lemma \ref{lemma:transvection} allows us to construct some important elements
of $\Torelli_{g,1}$.  First, if $[x]=0$, then $T_x$ acts trivially on $\HH_1(\Sigma_{g,1};\Z)$.  This
will happen exactly when $x$ bounds an embedded subsurface of $x$ (see Figure \ref{figure:generators}).
We will call such elements
of Torelli {\em separating twists}.  Next, the formula Lemma \ref{lemma:transvection} only depends
on the homology class of the simple closed curve.  Thus if $y$ and $y'$ are homologous, then $T_y$ and
$T_{y'}$ act
the same on $\HH_1(\Sigma_{g,1};\Z)$, so $T_y T_{y'}^{-1} \in \Torelli_{g,1}$.  If $y$ and $y'$
are disjoint and homologous, then their union bounds an
embedded subsurface (see Figure \ref{figure:generators}).  If in addition to being disjoint neither $y$ nor
$y'$ is separating, then we will call $T_y T_{y'}^{-1} \in \Torelli_{g,1}$ a {\em bounding pair map}.

\ParagraphHeading{Generating sets for Torelli.}
We have the following theorem.

\begin{theorem}
\label{theorem:torelligen}
For all $g \geq 1$, the group $\Torelli_{g,1}$ is generated by bounding pair maps and separating twists.
For $g \geq 3$, only bounding pair maps are needed.
\end{theorem}

\pagebreak
\begin{remarks}
\mbox{}
\begin{enumerate}
\item The fact that $\Torelli_{g,1}$ is generated by bounding pair maps and separating twists was originally
proven by Powell \cite{Powell}, using earlier work of Birman \cite{BirmanSiegel}.  This proof depended on
some heroic calculations in the symplectic group whose details were omitted from the published
papers.  More recently, Putman \cite{PutmanCutPaste} gave a modern proof using the curve complex.  Even
more recently, Hatcher and Margalit \cite{HatcherMargalit} have given an even shorter proof.
\item The fact that for $g \geq 3$ only bounding pair maps are needed is due to Johnson \cite{JohnsonBP}.  He
later proved a remarkable theorem which says that $\Torelli_{g,1}$ is finitely generated for $g \geq 3$
(see \cite{JohnsonFinite}).
The size of Johnson's generating set grows exponentially in $g$.
Answering a conjecture of Johnson, Putman \cite{PutmanSmallGensetTorelli} has recently constructed
a generating set for $\Torelli_{g,1}$ that grows cubically in $g$.  As we will discuss below, the
abelianization of $\Torelli_{g,1}$ has rank cubic in the genus, so one cannot do better.
\item McCullough and Miller \cite{McCulloughMiller} proved that $\Torelli_{2,n}$ is not
finitely generated for $n \in \{0,1\}$.  Later, in his thesis Mess \cite{MessThesis} proved
that $\Torelli_{2}$ is an infinite rank free group.
\item It is not known if $\Torelli_{g,n}$ is finitely presentable for $g \geq 3$.
\end{enumerate}
\end{remarks}

\ParagraphHeading{The Birman exact sequence.}
We will need to make several calculations in $\Torelli_{g,n}$.  For us, the key tool
for making such calculations is the fact that $\Torelli_{g,n}$ contains a large number
of groups derived from surface groups.  This follows from the {\em Birman exact
sequence}, which takes the following form.  Let $U\Sigma_{g}$ be the unit tangent bundle
of $\Sigma_g$.  For $g \geq 2$, we then have an exact sequence
$$1 \longrightarrow \pi_1(U\Sigma_g) \longrightarrow \Mod_{g,1} \longrightarrow \Mod_g \longrightarrow 1.$$
The terms here have the following meanings.  Let $\beta$ be the boundary component of $\Sigma_{g,1}$.
\begin{itemize}
\item The map $\Mod_{g,1} \rightarrow \Mod_g$ comes from gluing a disc to $\beta$ and
extending mapping classes over this disc by the identity.
\item The subgroup $\pi_1(U\Sigma_g)$ of $\Mod_{g,1}$ is known as the ``disc-pushing subgroup''.  The
mapping class associated to $\gamma \in \pi_1(U\Sigma_g)$ ``pushes'' the boundary component $\beta$ around
the surface while allowing it to rotate. 
\end{itemize}
Of course, the original version of the Birman exact sequence goes back to work of Birman \cite{BirmanBraids}.
The version here first appeared in \cite{JohnsonFinite}; see \cite{FarbMargalitPrimer} for a textbook treatment.

\Figure{figure:birmanexactseq}{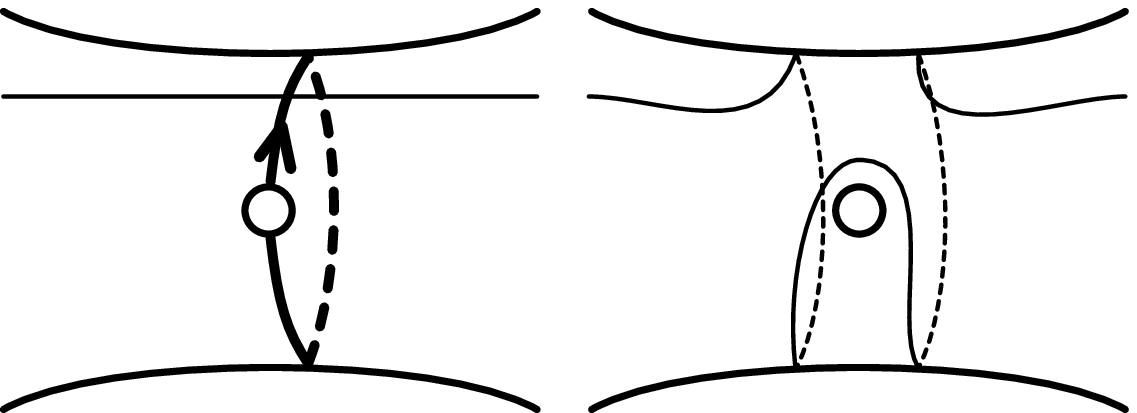}{Pushing the boundary component around a simple closed
curve induces a bounding pair map}

\ParagraphHeading{The loop around the fiber.}
The fiber $F_0$ of $\U\Sigma_g$ over the basepoint satisfies $F_0 \cong S^1$.  
The orientation on $\Sigma_g$ determines an orientation on $F_0$, so it makes
sense to talk about ``clockwise'' and ``counterclockwise'' directions on $F_0$.  
The group $\pi_1(U\Sigma_g)$ contains a distinguished element $\delta_0$ which goes once around
$F_0$ in the clockwise direction.  The element of the disc-pushing subgroup of $\Mod_{g,1}$
corresponding to $\delta_0$ rotates the boundary component $\beta$ by a full turn in the clockwise
direction.  Clearly this is simply $T_{\beta}$.  Observe that $T_{\beta} \in \Torelli_{g,1}$.

\ParagraphHeading{Calculating in the disc-pushing subgroup.}
Consider some $\gamma \in \pi_1(\Sigma_g)$ that can be realized by a smoothly embedded simple closed curve.
The derivative of a smooth 
simple representative of $\gamma$ is a path in the tangent bundle of $\gamma$ which does not contain any
zero vectors.  For some fixed Riemannian metric on the surface, we can reparametrize $\gamma$
so that its derivative is a loop $\tilde{\gamma}$ in the unit tangent bundle.  If $\gamma \neq 1$,
then the element $\tilde{\gamma} \in \pi_1(U\Sigma_g)$ is independent of the choice
of a smoothly embedded representative of $\gamma$.  Indeed, any two 
such realizations are smoothly homotopic (this can be proved using the techniques in
\cite{Epstein}; to test your understanding, you should verify that
this fails if $\gamma = 1$).  

Let $\tau_{\gamma} \in \Mod_{g,1}$ be the element of the disc-pushing subgroup corresponding
to $\tilde{\gamma} \in \pi_1(U\Sigma_g)$.
As is shown in Figure \ref{figure:birmanexactseq}, the mapping class $\tau_{\gamma}$
is a bounding pair map, and hence lies in $\Torelli_{g,1}$.  Since the loop $\delta_0$ around
the fiber also corresponds to an element of $\Torelli_{g,1}$, we deduce that the disc-pushing
subgroup lies in $\Torelli_{g,1}$.  This implies that 
relations in $\pi_1(U\Sigma_{g})$ yield relations in $\Torelli_{g,1}$.  Even
more relations can be obtained by embedding $\Torelli_{g,1}$ into $\Torelli_{g',n}$ via a subsurface
inclusion $\Sigma_{g,1} \hookrightarrow \Sigma_{g',n}$.

\Figure{figure:lantern}{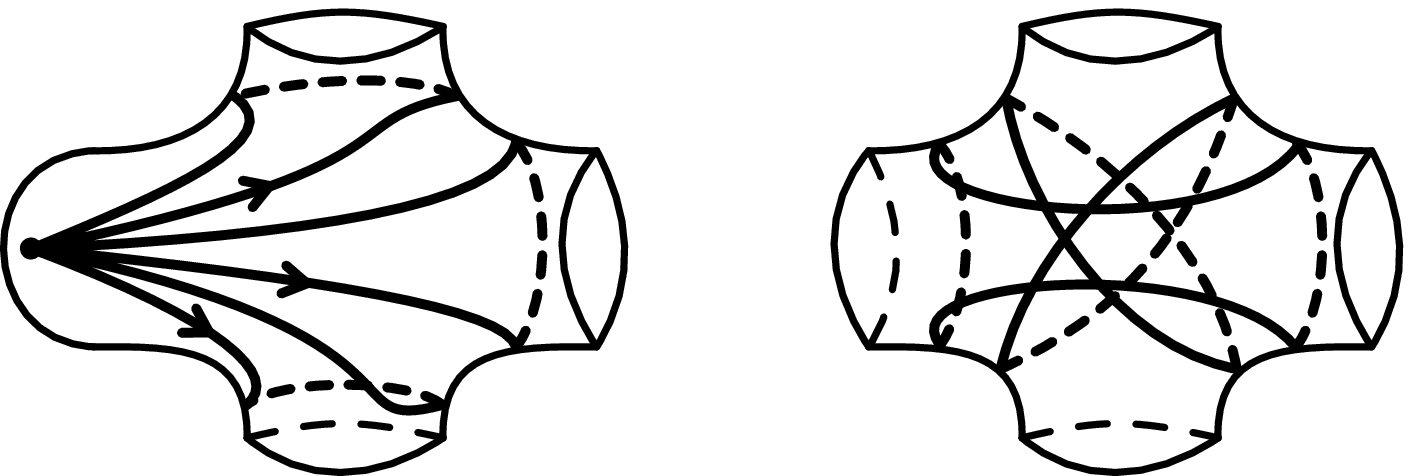}{The lantern relation is 
$\tilde{\gamma}_1 \cdot \tilde{\gamma}_2 \cdot \tilde{\gamma}_3 = \delta_0^k$.  To make the figure more
attractive, the curves $\gamma_i$ have what appears to be a singularity at the basepoint, but in
reality one should imagine them rounded and smooth there.  In terms of Dehn twists, the
lantern relation is $(T_{x_3} T_{y_3}^{-1}) (T_{x_2} T_{y_2}^{-1}) (T_{x_1} T_{y_1}^{-1}) = T_{\beta}$}

As an example, consider the relation $\gamma_1 \cdot \gamma_2 \cdot \gamma_3 = 1$ in $\pi_1(\Sigma_{g,1})$
depicted in Figure \ref{figure:lantern}.  We have
$\tilde{\gamma}_1 \cdot \tilde{\gamma}_2 \cdot \tilde{\gamma}_3 = \delta_0^k$ for some $k \in \Z$.

\begin{exercise}
Prove that $k = 1$.
\end{exercise}

\noindent
The associated relation $\tau_{\gamma_3} \tau_{\gamma_2} \tau_{\gamma_1} = \delta_0$ in $\Torelli_{g,1}$ is 
the {\em lantern relation}
\begin{equation}
\label{eqn:lantern}
(T_{x_3} T_{y_3}^{-1}) (T_{x_2} T_{y_2}^{-1}) (T_{x_1} T_{y_1}^{-1}) = T_{\beta};
\end{equation}
here the curves $x_i$ and $y_i$ are as depicted in Figure \ref{figure:lantern}.  

\begin{remark}
The order of the terms in \eqref{eqn:lantern} is the opposite of what one might expect because elements
in the fundamental group are composed left to right but mapping classes are composed right to left.
\end{remark}

\noindent
Observe that if
$g' \geq 3$, then this relation can be embedded in $\Torelli_{g',n}$ to express a separating
twist as a product of bounding pair maps (c.f.\ Theorem \ref{theorem:torelligen}).

\ParagraphHeading{Killing off separating twists.}
The lantern relation gives numerous ways of expressing $T_{\beta}$ as a product of bounding pair
maps.  For $1 \leq i \leq 3$, let $x_i,y_i \in \pi_1(\Sigma_g)$ be the curves in Figure \ref{figure:killsep}.
Observe that $x_1 x_2 x_3 = 1$ and $y_1 y_2 y_3=1$, so we have two different lantern relations
$$T_{\beta} = \tau_{x_3} \tau_{x_2} \tau_{x_1} \quad \text{and} \quad T_{\beta} = \tau_{y_3} \tau_{y_2} \tau_{y_1}.$$
These curves have the property that $x_i$ is homologous to $y_i^{-1}$ for $1 \leq i \leq 3$.  The group
$\Torelli_{g,1}$ acts on $\pi_1(\Sigma_g)$, and it is not hard to see that
there exists some $f_i \in \Torelli_{g,1}$ such that $f_i(x_i) = y_i^{-1}$.  To use this, we will need the following exercise.

\begin{exercise}
\label{exercise:conjugation}
If $\gamma \in \pi_1(\Sigma_g)$ can be realized by a simple closed curve and if $f \in \Mod_{g,1}$, then
$\tau_{f(\gamma)} = f \tau_{\gamma} f^{-1}$.
\end{exercise}

\noindent
Applying Exercise \ref{exercise:conjugation} several times, we obtain the following relation in $\Torelli_{g,1}$.
\begin{align}
T_{\beta}^2 &= (\tau_{x_3} \tau_{x_2} \tau_{x_1}) (\tau_{y_3} \tau_{y_2} \tau_{y_1})\notag\\
&= (\tau_{x_3} \tau_{x_2} \tau_{x_1}) (\tau_{f_3(x_3)}^{-1} \tau_{f_2(x_2)}^{-1} \tau_{f_1(x_1)}^{-1})\notag\\
&= (\tau_{x_3} \tau_{x_2} \tau_{x_1})(f_3 \tau_{x_3}^{-1} f_3^{-1} f_2 \tau_{x_2}^{-1} f_2^{-1}f_1 \tau_{x_1}^{-1} f_1^{-1})\label{eqn:killsep}
\end{align}
Upon abelianizing $\Torelli_{g,1}$, the right hand side of \eqref{eqn:killsep} vanishes.  Letting 
$[T_{\beta}] \in \HH_1(\Torelli_{g,1};\Z)$ be the associated element of the abelianization, we obtain
that $2 [T_{\beta}] = 0$.

\Figure{figure:killsep}{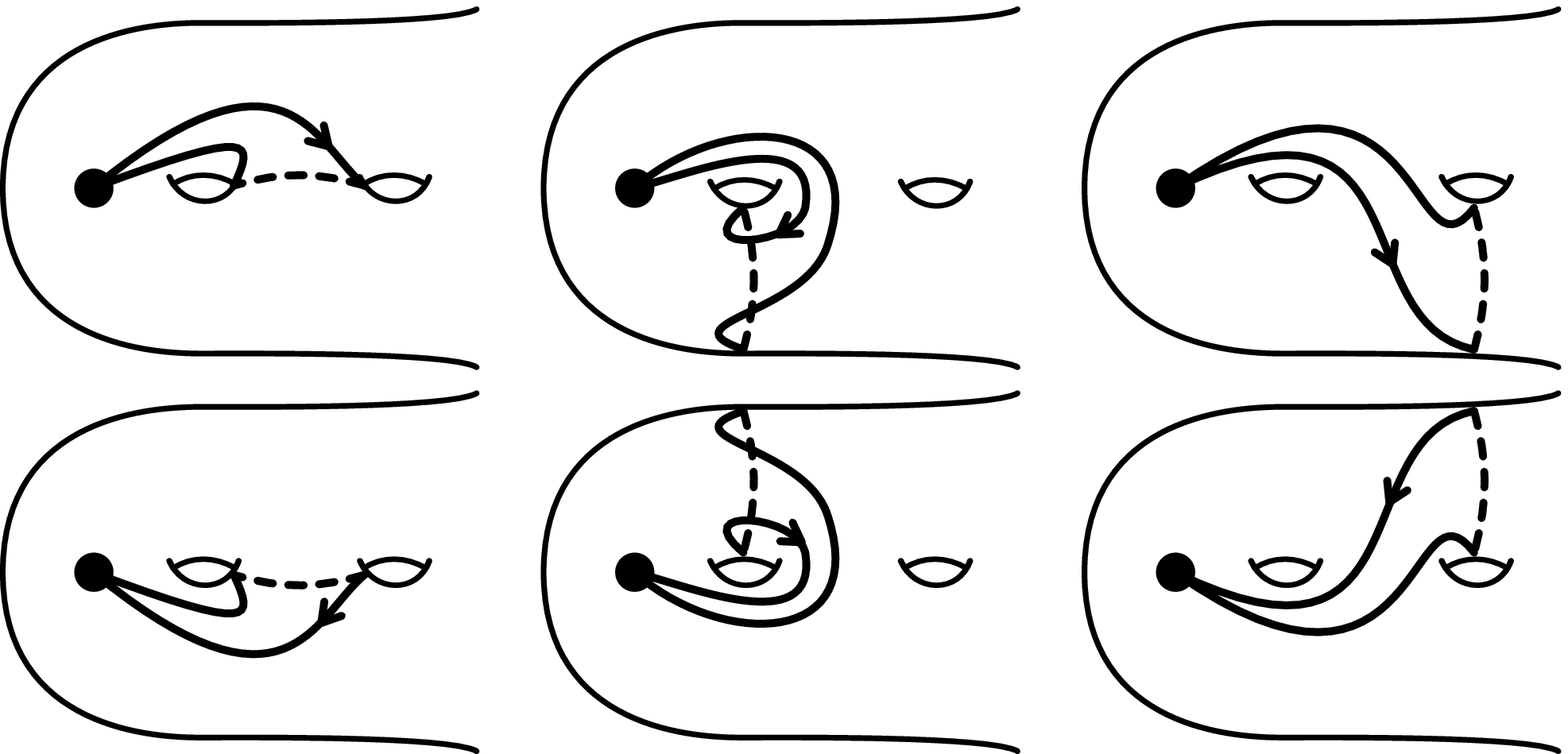}{$x_1 x_2 x_3 = 1$ and $y_1 y_2 y_3=1$}

If $x$ is a separating curve on a surface of genus at least $3$, then we can embed the above relation
into the surface to get that $T_x^2$ has to vanish upon abelianizing the Torelli group.  We have proven
the following.

\begin{lemma}
\label{lemma:septwiststorsion}
Fix $g \geq 3$ and $n \in \{0,1\}$.  Let $T_x$ be a separating twist in $\Torelli_{g,n}$.  Then
the image $[T_x]$ of $T_x$ in $\HH_1(\Torelli_{g,n};\Z)$ satisfies $2 [T_x]=0$.
\end{lemma}

\noindent
Lemma \ref{lemma:septwiststorsion} first appeared in \cite{JohnsonAbel}.  The above is a version of Johnson's
proof.  For an alternate exposition of that proof which arranges the details a little differently, see
\cite[\S 7.2]{PutmanJohnsonHomo}.

\ParagraphHeading{A preview.}
In Lecture 2, we will construct the important Johnson homomorphism.  Letting
$\HHH = \HH_1(\Sigma_{g,n};\Z)$, this is a surjective homomorphism
$$\tau : \Torelli_{g,1} \longrightarrow \wedge^3 \HHH.$$
There is also a version for closed surfaces, but we will not discuss it.  The key property
of the Johnson homomorphism is that its kernel is exactly the subgroup generated by separating
twists.  Lemma \ref{lemma:septwiststorsion} will then allow us to deduce the following theorem
of Johnson \cite{JohnsonAbel}.

\begin{theorem}
\label{theorem:torelliabel}
For $g \geq 3$, we have $\HH_1(\Torelli_{g,1};\Z) \cong W \oplus \wedge^3 \HHH$, where $W$ consists
of $2$-torsion.
\end{theorem}

\begin{remark}
Johnson also calculated the $2$-torsion $W$.  The associated $\Z/2$-quotients of
$\Torelli_{g,1}$ come from the Rochlin invariants of homology $2$-spheres.  They
were originally constructed by Birman and Craggs \cite{BirmanCraggs}.
Later, in \cite{JohnsonBirmanCraggs} Johnson packaged all of Birman and Craggs's homomorphisms together
into a single homomorphism and determined exactly how many linearly independent quotients
they had constructed.
\end{remark}

\noindent
As a prologue for the construction, we recommend performing the following exercise, which explains
the appearance of $\wedge^3 \HHH$ in the Johnson homomorphism.

\begin{exercise}
\label{exercise:torus}
Let $\Torus^n$ denote the $n$-torus $(S^1)^n$. 
\begin{enumerate}
\item Prove that the cohomology ring $\HH^{\ast}(\Torus^n;\Z)$ is 
isomorphic to the exterior algebra $\wedge^{\ast} \Z^n$.
\item Let $G$ be an abelian topological group.  Define a product 
$$\HH_i(G;\Z) \otimes \HH_j(G;\Z) \longrightarrow \HH_{i+j}(G;\Z)$$
via the composition
$$\HH_i(G;\Z) \otimes \HH_j(G;\Z) \stackrel{\phi}{\longrightarrow} \HH_{i+j}(G \times G;\Z) \stackrel{\psi}{\longrightarrow} \HH_{i+j}(G;\Z),$$
where $\phi$ is the map coming from the K\"{u}nneth exact sequence and $\psi$ is induced by the group
product $G \times G \rightarrow G$.  Prove that with this product structure,
$\HH_{\ast}(G;\Z)$ is a graded-commutative algebra.  We remark that this product
is known as the {\em Pontryagin product}.
\item The space $\Torus^n$ is an abelian topological group.  Prove that the 
resulting graded-commutative ring $\HH_{\ast}(\Torus^n;\Z)$ is isomorphic to the exterior
algebra $\wedge^{\ast} \Z^n$.
\end{enumerate}
\end{exercise}

\subsection{Appendix to lecture 1 : the surjectivity of the symplectic representation}
\label{appendix:symplectic}

Recall that the action of $\Mod_{g,1}$ on
$\HH_1(\Sigma_{g,1};\Z)$ preserves the algebraic intersection pairing and thus
gives a representation $\pi : \Mod_{g,1} \rightarrow \Sp_{2g}(\Z)$.  In this
appendix, we will give a sequence of exercises about the surjectivity of $\pi$.

A {\em symplectic basis} for $\HH_1(\Sigma_{g,1};\Z)$ is a
basis $\{a_1,b_1,\ldots,a_g,b_g\}$ for $\HH_1(\Sigma_{g,1};\Z)$ such that
$$i_a(a_i,b_j) = \delta_{ij} \quad \text{and} \quad i_a(a_i,a_j) = i_a(b_i,b_j) = 0$$
for all $1 \leq i,j \leq g$.  Let $\mathcal{S}$ be the set of symplectic bases for
$\HH_1(\Sigma_{g,1};\Z)$.  The following exercise should be straightforward.

\begin{exercise}
\label{exercise:sptran}
$\Sp_{2g}(\Z)$ acts simply transitively on $\mathcal{S}$.
\end{exercise}

\Figure{figure:symplecticbasis}{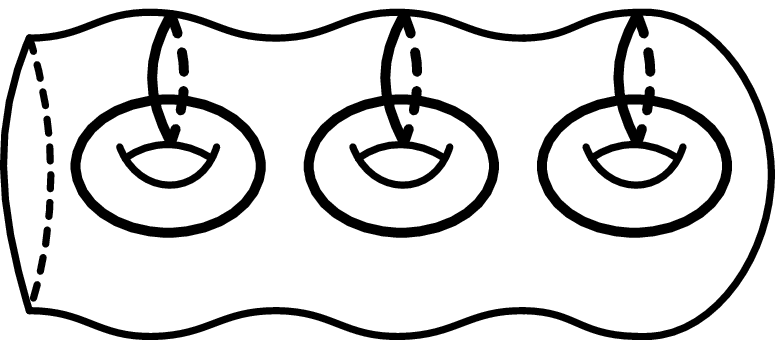}{A geometric symplectic basis}

If $x$ and $y$ are simple closed curves on $\Sigma_{g,1}$, then let $i_g(x,y)$ be their
{\em geometric intersection number}; i.e.\ the minimal cardinality of $x' \cap y'$ as
$x'$ and $y'$ range over all simple closed curves homotopic to $x$ and $y$, respectively.  A
{\em geometric symplectic basis} (see Figure \ref{figure:symplecticbasis})
is a collection $\{\alpha_1,\beta_1,\ldots,\alpha_g,\beta_g\}$
of simple closed curves on $\Sigma_{g,1}$ such that
$$i_g(\alpha_i,\beta_j) = \delta_{ij} \quad \text{and} \quad i_g(\alpha_i,\alpha_j) = i_g(\beta_i,\beta_j) = 0$$
for all $1 \leq i,j \leq g$.  Let $\mathcal{G}$ be the set of geometric symplectic bases on
$\Sigma_{g,1}$.  We then have the following.

\begin{exercise}
\label{exercise:modtran}
$\Mod_{g,1}$ acts transitively on $\mathcal{G}$.  Hint : given two geometric symplectic bases,
prove using the Euler characteristic that you get homeomorphic surfaces when you cut along them.
\end{exercise}

The following lemma is the heart of the fact that $\pi(\Mod_{g,1}) = \Sp_{2g}(\Z)$.

\begin{lemma}
\label{lemma:realizebasis}
If $\{a_1,b_1,\ldots,a_g,b_g\}$ is a symplectic basis for $\HH_1(\Sigma_g;\Z)$, then
there exists a geometric symplectic basis $\{\alpha_1,\beta_1,\ldots,\alpha_g,\beta_g\}$
on $\Sigma_g$ such that $[\alpha_i] = a_i$ and $[\beta_i] = b_i$ for $1 \leq i \leq g$.
\end{lemma}

\noindent
Proofs of Lemma \ref{lemma:realizebasis} can be found in \cite[Lemma A.3]{PutmanCutPaste} and
\cite[3rd proof of Lemma 6.4]{FarbMargalitPrimer}; however, it is worthwhile to contemplate
how one might prove it (though it is probably too hard for an exercise).

\begin{exercise}
Combine Lemma \ref{lemma:realizebasis} with Exercises \ref{exercise:sptran} and \ref{exercise:modtran}
to deduce that $\pi(\Mod_{g,1}) = \Sp_{2g}(\Z)$.
\end{exercise}

\section*{Lecture 2 : The Johnson homomorphism}
\addtocounter{section}{1}
\setcounter{theorem}{0}
\setcounter{figure}{0}

Let $\HHH = \HH_1(\Sigma_g;\Z)$.  In this lecture, we will construct the {\em Johnson homomorphism},
which is a surjective homomorphism
$$\tau : \Torelli_{g,1} \longrightarrow \wedge^3 \HHH.$$
This homomorphism can be constructed in a number of completely different ways.  It was originally
constructed in \cite{JohnsonHomo} by examining the action of $\Torelli_{g,1}$ on the second
nilpotent truncation of $\pi_1(\Sigma_{g,1})$.  We explain this original construction
in an appendix.  In his survey \cite{JohnsonSurvey}, Johnson
outlined several alternate constructions.  We will use a definition in terms of mapping tori
which was introduced in \cite{JohnsonSurvey} and was first shown to be equivalent to the
original definition by Hain \cite{HainTorelli}.  Our exposition will follow
the paper \cite{ChurchFarbTorelli} of Church and Farb, which gives a more direct proof of this equivalence.

\ParagraphHeading{The construction.}
Consider $f \in \Torelli_{g,1}$.  Though it is an abuse of notation, we will regard $f$ as
a homeomorphism of $\Sigma_{g,1}$.  Glue a disc to the boundary component of $\Sigma_{g,1}$ and
extend $f$ over this disc by the identity to obtain a homeomorphism $F$ of $\Sigma_g$.  Let
$p_0 \in \Sigma_g$ be the center of the glued-in disc, so $F(p_0)=p_0$.  Now let
$M_{F}$ be the {\em mapping torus} of $F$, i.e.\ the
quotient $\Sigma_g \times I / \sim$, where $(x,1) \sim (F(x),0)$.  Give $M_F$ the
basepoint $q_0 = (p_0,0)$.  

There is a distinguished element $t \in \pi_1(M_F,q_0)$
which traverses the embedded loop $p_0 \times I / \sim$ in $M_F$ in the positive direction.
Fix a standard generating set $S=\{s_1,\ldots,s_{2g}\}$ for $\pi_1(\Sigma_g)$ that
satisfies the surface relation
$$[s_1,s_2] \cdots [s_{2g-1},s_{2g}] = 1.$$
Since $F(p_0)=p_0$, the map $F$ acts on $\pi_1(\Sigma_g,p_0)$.  For $1 \leq i \leq 2g$, let $w_i$ be an expression
for $F_{\ast}(s_i)$ in terms of the generating set $S$.  We then have a presentation
$$\pi_1(M_F,q_0) = \GroupPres{s_1,\ldots,s_{2g},t}{[s_1,s_2] \cdots [s_{2g-1},s_{2g}]=1,
t s_i t^{-1} = w_i \text{ for } 1 \leq i \leq 2g}.$$
For $\gamma \in \pi_1(\Sigma_g,p_0)$, let $[\gamma] \in \HHH$ be the associated element
of the abelianization.  Since $F \in \Torelli_{g}$, we have $[s_i] = [w_i]$ for $1 \leq i \leq 2g$.
This implies that we can define a homomorphism $\phi_{\ast} : \pi_1(M_F,q_0) \rightarrow \HHH$
such that $\phi_{\ast}(s_i) = [s_i]$ for $1 \leq i \leq 2g$ and such that $\phi_{\ast}(t)=0$.

The space $M_{F}$ is clearly a $K(\pi_1(M_{F}),1)$.  Let $\Torus^{2g}$ be the $2g$-torus.  Fix
an identification of $\pi_1(\Torus^{2g})$ with $\HHH$.
Since $\Torus^{2g}$ is a $K(\HHH,1)$, the standard properties of Eilenberg-MacLane spaces
show that there is a canonical homotopy class of continuous maps $\phi : M_{F} \rightarrow \Torus^{2g}$
inducing the homomorphism $\phi_{\ast}$.
The space $M_F$ is a closed 3-manifold, so it has a canonical class $[M_F] \in \HH_3(M_F;\Z)$.  Define
$$\tau(f) = \phi_{\ast}([M_{F}]) \in \HH_3(\Torus^{2g};\Z) \cong \wedge^3 \HHH.$$
The last isomorphism here comes from Exercise \ref{exercise:torus}

Summing up, we have constructed a map $\tau : \Torelli_{g,1} \rightarrow \wedge^3 \HHH$.  The
following exercise is a good test of your understanding of the above construction.

\begin{exercise}
\label{exercise:tau}
Prove that $\tau$ is independent of all the above choices except for the identification
of $\pi_1(\Torus^{2g})$ with $\HHH$ (which is fixed).  Next, prove that $\tau$ is a homomorphism.
\end{exercise}

\ParagraphHeading{Effect on generators.}
The following lemma calculates $\tau$ on the generators for $\Torelli_{g,1}$.  Recall that
if $S$ is a genus $h$ surface with at most $1$ boundary component, then a {\em symplectic basis} for
$\HH_1(S;\Z) \cong \Z^{2h}$ is a basis $\{a_1,b_1,\ldots,a_h,b_h\}$ for $\HH_1(S;\Z)$ such that
$$i_a(a_i,b_j) = \delta_{ij} \quad \text{and} \quad i_a(a_i,a_j) = i_a(b_i,b_j) = 0$$
for all $1 \leq i,j \leq h$.  Here $i_a(\cdot,\cdot)$ is the algebraic intersection form.

\begin{lemma}
\label{lemma:johnsoncalc}
\mbox{}
\begin{enumerate}
\item Let $T_x \in \Torelli_{g,1}$ be a separating twist.  Then $\tau(T_x) = 0$.
\item Let $T_{x} T_{x'}^{-1}$ be a bounding pair map on $\Torelli_{g,1}$.  Let $S$ be the component of
$\Sigma_{g,1}$ cut along $x \cup x'$ that does not contain $\partial \Sigma_{g,1}$, so $S \cong \Sigma_{h,2}$
for some $h < g$.  Let $S' \subset S$ be an embedded subsurface such that $S' \cong \Sigma_{h,1}$ and
let $\{a_1,b_1,\ldots,a_h,b_h\}$ be a symplectic basis for $\HH_1(S';\Z) \subset \HH_1(\Sigma_{g,1};\Z)$.
Then
$$\tau(T_x T_{x'}^{-1}) = \pm [x] \wedge (a_1 \wedge b_1 + \cdots + a_h \wedge b_g).$$
\end{enumerate}
\end{lemma}

\noindent
We will discuss the proof of this lemma at the end of this section.  Right now, we suggest doing the following
two exercises.

\begin{exercise}
Prove that the formula in Lemma \ref{lemma:johnsoncalc} is independent of the choice of $S'$ and its
symplectic basis.
\end{exercise}

\begin{exercise}
Using Lemma \ref{lemma:johnsoncalc}, prove that $\tau$ is surjective.
\end{exercise}

\ParagraphHeading{Johnson's theorem.}
In \cite{JohnsonKg}, Johnson proved the following deep theorem, which is a sort of converse
to part 1 of Lemma \ref{lemma:johnsoncalc}.  For an alternate proof, see \cite{PutmanJohnsonHomo}.

\begin{theorem}
\label{theorem:kg}
The kernel of $\tau$ is generated by separating twists.
\end{theorem}

\noindent
As we indicated at the end of Lecture 1, this theorem together with Lemma \ref{lemma:septwiststorsion} implies
that $\HH_1(\Torelli_{g,1};\Z) \cong W \oplus \wedge^3 \HHH$ for $g \geq 3$, where $W$ consists 
of $2$-torsion.

\ParagraphHeading{The Johnson homomorphism mod $p$.}
Set $\HHH_p = \HH_1(\Sigma_{g,1};\Z/p)$.  We wish to construct a ``mod $p$'' Johnson homomorphism
$$\tau_p : \Mod_{g,n}(p) \rightarrow \HHH_p.$$
The construction goes exactly like the construction of the ordinary Johnson homomorphism.
Consider $f \in \Mod_{g,1}(p)$.  Regard $f$ as
a homeomorphism of $\Sigma_{g,1}$.  Glue a disc to the boundary component of $\Sigma_{g,1}$ and
extend $f$ over this disc by the identity to obtain a homeomorphism $F$ of $\Sigma_g$.  Let
$p_0 \in \Sigma_g$ be the center of the glued-in disc, so $F(p_0)=p_0$.  Let $M_F$ be
the mapping torus of $F$ and let $q_0 = (p_0,0)$ be the basepoint for $M_F$.

There is a distinguished element $t \in \pi_1(M_F,q_0)$
which traverses the embedded loop $p_0 \times I / \sim$ in $M_F$ in the positive direction.
Fix a standard generating set $S=\{s_1,\ldots,s_{2g}\}$ for $\pi_1(\Sigma_g)$ that
satisfies the surface relation
$$[s_1,s_2] \cdots [s_{2g-1},s_{2g}] = 1.$$
Since $F(p_0)=p_0$, the map $F$ acts on $\pi_1(\Sigma_g,p_0)$.  For $1 \leq i \leq 2g$, let $w_i$ be an expression
for $F_{\ast}(s_i)$ in terms of the generating set $S$.  We then have a presentation
$$\pi_1(M_F,q_0) = \GroupPres{s_1,\ldots,s_{2g},t}{[s_1,s_2] \cdots [s_{2g-1},s_{2g}]=1,
t s_i t^{-1} = w_i \text{ for } 1 \leq i \leq 2g}.$$
For $\gamma \in \pi_1(\Sigma_g,p_0)$, let $[\gamma]_p \in \HHH_p$ be the associated element.
Since $F \in \Mod_{g}(p)$, we have $[s_i]_p = [w_i]_p$ for $1 \leq i \leq 2g$.
This implies that we can define a homomorphism $\phi_{\ast} : \pi_1(M_F,q_0) \rightarrow \HHH_p$
such that $\phi_{\ast}(s_i) = [s_i]_p$ for $1 \leq i \leq 2g$ and such that $\phi_{\ast}(t)=0$.

The space $M_{F}$ is clearly a $K(\pi_1(M_{F}),1)$.  Let $Z$ be a $K(\HHH_p,1)$.
The standard properties of Eilenberg-MacLane spaces
show that there is a canonical homotopy class of continuous maps $\phi : M_{F} \rightarrow Z$
inducing the homomorphism $\phi_{\ast}$.
Define
$$\tau_p'(f) = \phi_{\ast}([M_{F}]) \in \HH_3(Z;\Z) \cong \HH_3(\HHH_p;\Z),$$
where $M_F \in \HH_3(M_F;\Z)$ is the canonical class.
It is not true that $\HH_3(\HHH_p;\Z) \cong \wedge^3 \HHH_p$; however,
$\HH_3(\HHH_p;\Z)$ does contain $\wedge^3 \HHH_p$ as a direct factor (see, for example,
\cite[Theorem V.6.4]{BrownCohomology}).  Let $\tau_p : \Mod_{g,1}(p) \rightarrow \wedge^3 \HHH_p$
be the composition of $\tau_p'$ with some (fixed for all time)
projection of $\HH_3(\HHH_p;\Z)$ onto $\wedge^3 \HHH_p$.
By the same argument used in Exercise \ref{exercise:tau}, the map $\tau_p$ is a well-defined
homomorphism.

From its construction, it is clear that the following diagram commutes.
$$\begin{CD}
\Torelli_{g,1} @>{\tau}>> \wedge^3 \HHH \\
@VVV                      @VVV \\
\Mod_{g,n}(p)  @>{\tau_p}>> \wedge^3 \HHH_p
\end{CD}$$
Since $\tau$ is surjective for $g \geq 3$, we obtain the following.

\begin{theorem}
\label{theorem:modpjohnson}
The map $\tau_p : \Mod_{g,1}(p) \rightarrow \HHH_p$ is surjective for $g \geq 3$.
\end{theorem}

\ParagraphHeading{Images of mapping tori.}
To prove Lemma \ref{lemma:johnsoncalc}, one would need to be able to calculate 
the images of fundamental classes of mapping tori under maps to a torus.  Instead
of giving the details of the proof, we will prove an easier theorem which illustrates
how one can do this.  For the proof of Lemma \ref{lemma:johnsoncalc}, see
\cite{ChurchFarbTorelli}.

The result we will prove is as follows.  In its statement and proof, we will use the
natural identification of the graded-commutative algebra structure on $\HH_{\ast}(\Torus^n;\Z)$ 
with $\wedge^{\ast} \Z^n$ which was identified in Exercise \ref{exercise:torus}.  

\begin{theorem}
Let $\phi : \Sigma_g \rightarrow \Torus^n$ be a continuous map.  Choose a standard basis
$\{\alpha_1,\beta_1,\ldots,\alpha_g,\beta_g\}$ for $\pi_1(\Sigma_g)$, so
$$\pi_1(\Sigma_g) = \GroupPres{\alpha_1,\beta_1,\ldots,\alpha_g,\beta_g}{[\alpha_1,\beta_1]\cdots[\alpha_g,\beta_g]=1}.$$
Then
$$\phi_{\ast}([\Sigma_g]) = \pm \sum_{i=1}^g \phi_{\ast}(\alpha_i) \wedge \phi_{\ast}(\beta_i) \in \wedge^2 \Z^n \cong \HH_2(\Torus^n;\Z).$$
\end{theorem}
\begin{proof}
Assume first that $g=1$, so $\Sigma_g \cong \Torus^2$.  We will identify $\Sigma_g$ with
$\Torus^2$.  Observe that
$\pi_1(\Torus^2) \cong \Z^2$ and $\HH_{\ast}(\Torus^2;\Z) \cong \wedge^{\ast}(\Z^2)$.  Moreover,
$\HH_2(\Torus^2;\Z) \cong \wedge^2 \Z^2$ is generated by $\alpha_1 \wedge \beta_1$, so 
$[\Torus^2] = \pm \alpha_1 \wedge \beta_1$.  By
the naturality of the Pontryagin product we have
$$\phi_{\ast}([\Torus^2]) = \pm \phi_{\ast}(\alpha_1) \wedge \phi_{\ast}(\beta_1).$$

\Figure{figure:calcjohnson}{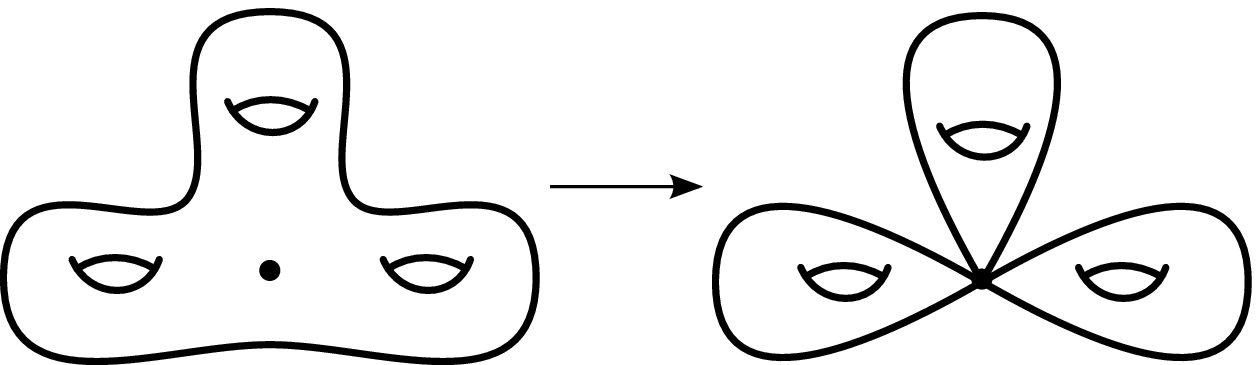}{Degenerating our surface to a wedge of tori}

Now assume that $g > 1$.  We will reduce to the case $g = 1$ by a method that will
be familiar to algebraic geometers.  Namely, we will ``degenerate'' our surface
to a nodal surface and make the computation there.  Fix a basepoint $p_0$ for $\Sigma_g$.
As shown in Figure \ref{figure:calcjohnson}, let $(X,p_0')$
be the result of collapsing $(\Sigma_g,p_0)$ to the wedge of $g$ tori $X_1,\ldots,X_g$.  Choose $X$
such that under the collapse map $\rho : \Sigma_g \rightarrow X$, the curves $\alpha_i$ and $\beta_i$ map
to generators for $\pi_1(X_i,p_0') \subset \pi_1(X,p_0')$.  

\begin{exercise}
Prove that the map $\phi_{\ast} : \pi_1(\Sigma_g,p_0) \rightarrow \pi_1(\Torus^n)$ factors as
$$\pi_1(\Sigma_g,p_0) \stackrel{\rho_{\ast}}{\longrightarrow} \pi_1(X,p_0') \stackrel{\phi'_{\ast}}{\longrightarrow} \pi_1(\Torus^n)$$
for some homomorphism $\phi_{\ast}' : \pi_1(X,p_0') \rightarrow \pi_1(\Torus^n)$.  Hint : $\pi_1(\Torus^n)$
is abelian.
\end{exercise}

\noindent
Since the spaces $\Sigma_g$ and $X$ and $\Torus^n$ are all Eilenberg-MacLane spaces, there exists some continuous
map $\phi' : X \rightarrow \Torus^n$ inducing $\phi'_{\ast}$ such that $\phi$ is homotopic
to $\phi' \circ \rho$.

Observe that $\rho_{\ast}([\Sigma_g]) = [X_1] + \cdots + [X_g]$.  Moreover,
$\rho_{\ast}(\alpha_i)$ and $\rho_{\ast}(\beta_i)$ are a standard basis for 
$\pi_1(X_i) \subset \pi_1(X)$.  By the
case $g=1$, we have
$$(\phi'|_{X_i})_{\ast}([X_i]) = \pm \phi_{\ast}'(\rho_{\ast}(\alpha_i)) \wedge \phi_{\ast}'(\rho_{\ast}(\beta_i) = \pm \phi_{\ast}(\alpha_i) \wedge \phi_{\ast}(\beta_i).$$
All the $\pm$ signs are identical.  Adding everything up, we get the desired result.
\end{proof}

\section*{Appendix to lecture 2 : the original construction of the Johnson homomorphism}

The original construction of the Johnson homomorphism gave a homomorphism
$$\tau : \Torelli_{g,1} \rightarrow \Hom(\HHH,\wedge^2 \HHH).$$
Johnson then calculated $\tau$ on generators and showed that its image was isomorphic
to $\wedge^3 \HHH$.  In this appendix, we will guide you through Johnson's original
construction.

If $G$ is a group, then let $\gamma_k(G)$ be the $k^{\text{th}}$ term in its {\em lower
central series}, which is the inductively defined sequence of groups
$$\gamma_0(G) = G \quad \text{and} \quad \gamma_{k+1}(G) = [\gamma_k(G),G].$$
Set $\pi = \pi_1(\Sigma_{g,1},\ast)$, where $\ast \in \partial \Sigma_{g,1}$.
We will need the following fact.  We will state it as an exercise, but you should only
attempt it if you know some of the basics of group cohomology (otherwise, treat it
as a black box).

\begin{exercise}
Prove that $\gamma_1(\pi) / \gamma_2(\pi) \cong \wedge^2 \HHH$.  Hint : Apply the 5-term
exact sequence in group homology to the short exact sequence
$$1 \longrightarrow \gamma_1(\pi) \longrightarrow \pi \longrightarrow \HHH \longrightarrow 1.$$
You will need the fact that $\HH_2(\HHH;\Z) \cong \wedge^2 \HHH$.
\end{exercise}

Now consider the short exact sequence
$$1 \longrightarrow \gamma_1(\pi) \longrightarrow \pi \longrightarrow \HHH \longrightarrow 1.$$
To simplify this, we mod out by $\gamma_2(\pi) < \gamma_1(\pi)$.  By the above exercise,
we get a short exact sequence
\begin{equation}
\label{eqn:fundexact}
1 \longrightarrow \wedge^2 \HHH \longrightarrow \Gamma \longrightarrow \HHH \longrightarrow 1,
\end{equation}
where $\Gamma = \pi / \gamma_2(\pi)$.  
\begin{exercise}
The subgroup $\HHH < \Gamma$ is central.
\end{exercise}
The mapping class group acts on $\pi$ (this is where
we use the fact that the basepoint is on the boundary component, so it is fixed by the mapping
class group).  This action preserves $\gamma_k(\pi)$ for all $\pi$ (indeed, all automorphisms
of $\pi$ do this!).  We thus get an action of $\Mod_{g,1}$ on $\Gamma$ which preserves
$\wedge^2 \HHH < \Gamma$.

Restrict this action to $\Torelli_{g,1}$.  The group $\Torelli_{g,1}$ acts trivially on $\HHH$.
\begin{exercise}
The action of $\Torelli_{g,1}$ on $\wedge^2 \HHH < \Gamma$ is trivial.  This looks obvious,
but you have to trace through the above definitions to see that the action of $\Mod_{g,1}$
on $\wedge^2 \HHH < \Gamma$ is what you think it is.
\end{exercise}
Fix $f \in \Torelli_{g,1}$.  For $x \in \Gamma$, observe that $f(x)$ and $x$ project to the same
element of $\HHH$, so $x (f(x))^{-1} \in \wedge^2 \HHH$.  Define a set map 
$J_f' : \Gamma \rightarrow \wedge^2 \HHH$ by $J_f'(x) = x (f(x))^{-1}$.
\begin{exercise}
The map $J_f'$ factors through a set map $J_f : \HHH \rightarrow \wedge^2 \HHH$.
\end{exercise}
\begin{exercise}
The set map $J_f$ is a homomorphism.
\end{exercise}
We can thus define a set map
$$\tau : \Torelli_{g,1} \rightarrow \Hom(\HHH,\wedge^2 \HHH)$$
by $\tau(f) = J_f$.
\begin{exercise}
$\tau$ is a homomorphism.
\end{exercise}

\section*{Lecture 3 : The abelianization of $\Mod_{g,n}(p)$}
\addtocounter{section}{1}
\setcounter{theorem}{0}
\setcounter{figure}{0}

In this lecture, we calculate $\HH_1(\Mod_{g,n}(p);\Z)$.
To simplify our exposition, we will do the following.
\begin{itemize}
\item We will only consider the case of $\Mod_{g,1}(p)$.  The case of $\Mod_g(p)$ can
be dealt with in a similar way, but there are a few added complications.
\item We will only consider the case where $p$ is odd.  This greatly simplifies
both the statements of the results and their proofs.
\end{itemize}

\ParagraphHeading{Main theorem.}
Our main theorem is as follows.

\begin{theorem}
\label{theorem:modpabel}
Fix $g,p \geq 3$ such that $p$ is odd.  Set $\HHH_p = \HH_1(\Sigma_{g,1};\Z/p)$.  There
is then a short exact sequence
$$0 \longrightarrow \wedge^3 \HHH_p \longrightarrow \HH_1(\Mod_{g,1}(p);\Z) \longrightarrow \HH_1(\Sp_{2g}(\Z,p);\Z) \longrightarrow 0.$$
\end{theorem}

\begin{remarks}
\mbox{}
\begin{enumerate}
\item Theorem \ref{theorem:modpabel} was proven independently by Putman \cite{PutmanModPAbel}
and Sato \cite{SatoModPAbel}.  At the same time, Perron \cite{PerronModPAbel} calculated
$\HH_1(\Mod_{g,1}(p);\Z)$ up to a $2$-torsion ambiguity.
\item Sato \cite{SatoModPAbel} also calculated $\HH_1(\Mod_{g,1}(2);\Z)$.  The answer is more
complicated than that given in Theorem \ref{theorem:modpabel}.  This was later extended by
Putman \cite{PutmanPicardGroupLevel} to $\HH_1(\Mod_{g,1}(p);\Z)$ for $p$ not divisible by $4$.
The case where $p$ is divisible by $4$ is still open -- see the introduction of
\cite{PutmanPicardGroupLevel} for a discussion.
\item It was originally proven by Hain \cite{HainTorelli} that $\HH_1(\Mod_{g,1}(p);\Z)$ is finite
for $g \geq 3$.
\item At the end of this lecture, we will calculate $\HH_1(\Sp_{2g}(\Z,p);\Z)$.
\end{enumerate}
\end{remarks}

\ParagraphHeading{Beginning of the proof.}
Of course, Theorem \ref{theorem:modpabel} is derived from the short exact sequence
$$1 \longrightarrow \Torelli_{g,1} \longrightarrow \Mod_{g,1}(p) \longrightarrow \Sp_{2g}(\Z,p) \longrightarrow 1.$$
We will need the following exercise.

\begin{exercise}
\label{exercise:short}
Consider a short exact sequence
$$1 \longrightarrow G_1 \longrightarrow G_2 \longrightarrow G_3 \longrightarrow 1$$
of groups.  There is then a short exact sequence
$$0 \longrightarrow V \longrightarrow \HH_1(G_2;\Z) \longrightarrow \HH_1(G_3;\Z) \longrightarrow 0,$$
where $V$ is the image of $\HH_1(G_1;\Z)$ in $\HH_1(G_2;\Z)$.
\end{exercise}

\noindent
Applying Exercise \ref{exercise:short}, we see that 
it is enough to prove that the image of $\HH_1(\Torelli_{g,1};\Z)$ in
$\HH_1(\Mod_{g,1}(p);\Z)$ is $\wedge^3 \HHH_p$.  The key to this is the following lemma.

\begin{lemma}
\label{lemma:killl}
For $g \geq 3$, let $v \in \HH_1(\Mod_{g,1}(p);\Z)$ be in the image
of $\HH_1(\Torelli_{g,1};\Z)$.  Then $p \cdot v = 0$.
\end{lemma}

Before proving Lemma \ref{lemma:killl}, let use it to finish the proof of Theorem
\ref{theorem:modpabel}.  Theorem \ref{theorem:torelliabel} says that
$$\HH_1(\Torelli_{g,1};\Z) \cong W \oplus \wedge^3 \HHH,$$
where $W$ consists of $2$-torsion.  Lemma \ref{lemma:killl} and the fact that $p$ is odd
imply that the inclusion map $\HH_1(\Torelli_{g,1};\Z) \rightarrow \HH_1(\Mod_{g,1}(p);\Z)$ factors
through
$$(W / p \cdot W) \oplus (\wedge^3 \HHH)/(p \cdot \wedge^3 \HHH) = \wedge^3 \HHH_p.$$
Theorem \ref{theorem:modpjohnson} says that the image of $\HH_1(\Torelli_{g,1};\Z)$ in
$\HH_1(\Mod_{g,1}(p);\Z)$ contains $\wedge^3 \HHH_p$, so we conclude that the image equals
$\HHH_p$, as desired.

\ParagraphHeading{The crossed lantern relation.}
To prove Lemma \ref{lemma:killl}, we will need the following relation 
the mapping class group, which is known as the {\em crossed lantern relation}.

\Figure{figure:crossedlantern}{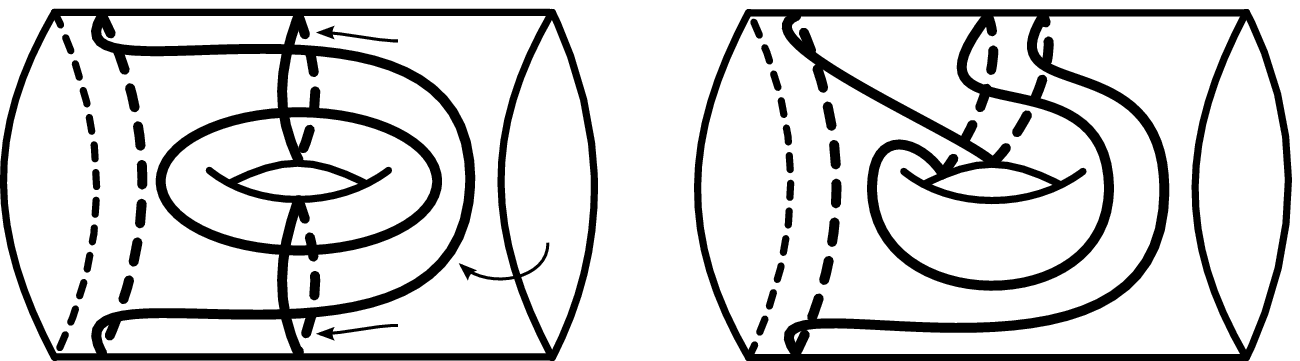}{The crossed lantern relation $(T_{y_1} T_{y_2}^{-1})(T_{x_1} T_{x_2}^{-1}) = T_{z_1} T_{z_2}^{-1}$.}

\begin{lemma}
\label{lemma:crossedlantern}
Let $x_i$ and $y_i$ and $z_i$ be the curves in Figure \ref{figure:crossedlantern}.
Then 
$$(T_{y_1} T_{y_2}^{-1})(T_{x_1} T_{x_2}^{-1}) = T_{z_1} T_{z_2}^{-1}.$$
\end{lemma}
\begin{proof}
This rests on two key facts:
$$T_{x_2}(y_i) = z_i \quad \text{and} \quad T_{y_2} T_{y_1}^{-1}(x_2) = x_1.$$
Repeatedly using the fact that $f T_{c} f^{-1} = T_{f(c)}$ for any simple closed curve $c$
and any mapping class $f$ (which can be proven exactly like Exercise \ref{exercise:conjugation}), we calculate
\begin{align*}
T_{x_1} T_{x_2}^{-1} &= \big((T_{y_2} T_{y_1}^{-1}) T_{x_2} (T_{y_1} T_{y_2}^{-1})\big) T_{x_2}^{-1} \\
                     &= (T_{y_2} T_{y_1}^{-1}) \big(T_{x_2} (T_{y_1} T_{y_2}^{-1}) T_{x_2}^{-1})\big) \\
                     &= T_{y_2} T_{y_1}^{-1} T_{z_1} T_{z_2}^{-1}.
\end{align*}
Rearranging this gives the desired relation.
\end{proof}

\begin{exercise}
Give an alternate derivation of the crossed lantern relation by the same technique we
used to construct the lantern relation in Lecture 1.  Hint : The lantern relation
comes from a relation $x y z = 1$ between simple closed curve in a surface group.  Try
to write down a different relation of the form $x y = z$ between simple closed curves.
If you get stuck, see \cite[\S 3.1.3]{PutmanInfinite}.
\end{exercise}

\ParagraphHeading{The proof of Lemma \ref{lemma:killl}.}
We are now in a position to prove Lemma \ref{lemma:killl}.
For $f \in \Mod_{g,1}(p)$, let $[f]$ be the corresponding element of $\HH_1(\Mod_{g,1}(p);\Z)$.
Theorem \ref{theorem:torelligen} says that $\Torelli_{g,1}$ is generated by
bounding pair maps.  Letting $T_{x_1} T_{x_2}^{-1}$ be a bounding pair map, it is thus
enough to show that $p \cdot [T_{x_1} T_{x_2}^{-1}] = 0$.  There exists an embedded subsurface
$S$ of $\Sigma_{g,1}$ containing $\{x_1,x_2\}$ such that $S \cong \Sigma_{1,2}$ and
such that the curves $\{x_1,x_2\}$ are embedded in $S$ as depicted in Figure \ref{figure:crossedlantern}.  Let
$\{y_1,y_2\}$ and $\{z_1,z_2\}$ be the curves in $S$ depicted in Figure \ref{figure:crossedlantern}.
We thus have a crossed lantern relation
\begin{equation}
\label{eqn:crossedlantern}
(T_{y_1} T_{y_2}^{-1}) (T_{x_1} T_{x_2}^{-1}) = T_{z_1} T_{z_2}^{-1}
\end{equation}
Also, $z_i = T_{x_2}(y_i)$ for $i=1,2$.

The key observation is that for all $k \geq 0$, conjugating \eqref{eqn:crossedlantern}
by $T_{x_2}^k$ results in another crossed lantern relation
$$(T_{T_{x_2}^k(y_1)} T_{T_{x_2}^k(y_2)}^{-1}) (T_{x_1} T_{x_2}^{-1}) 
= (T_{T_{x_2}^{k+1}(y_1)} T_{T_{x_2}^{n+1}(y_2)}^{-1}).$$
Since $T_{x_2}^p \in \Mod_{g,1}(p)$, we conclude that $[T_{y_1} T_{y_2}^{-1}]$ is equal in $\HH_1(\Mod_{g,1}(p);\Z)$ to
\begin{align*}
[T_{x_2}^p (T_{y_1} T_{y_2}^{-1}) T_{x_2}^{-p}] 
&= [(T_{T_{x_2}^p(y_1)} T_{T_{x_2}^p(y_2)}^{-1})]\\
&= [T_{x_1} T_{x_2}^{-1}] + [(T_{T_{x_2}^{p-1}(y_1)} T_{T_{x_2}^{p-1}(y_2)}^{-1})]\\
&= 2[T_{x_1}T_{x_2}^{-1}] + [(T_{T_{x_2}^{p-2}(y_1)} T_{T_{x_2}^{p-2}(y_2)}^{-1})]\\
&\hspace{5.5pt}\vdots\\
&= p\cdot[T_{x_1}T_{x_2}^{-1}] + [T_{y_1} T_{y_2}^{-1}],
\end{align*}
so $p\cdot[T_{x_1}T_{x_2}^{-1}] = 0$, as desired.

\ParagraphHeading{The abelianization of $\Sp_{2g}(\Z,p)$.}
To complete our description of the abelianization of $\Mod_{g,1}(p)$,
we need a description
of $\HH_1(\Sp_{2g}(\Z,p);\Z)$.  We first need some notation.  Denote the
$n \times n$ identity matrix by $\One_n$ and the $n \times n$ zero matrix
by $\Zero_n$.  Let $\Omega_g$ be the $2g \times 2g$ matrix 
$\MatTwoTwo{\Zero_g}{\One_g}{-\One_g}{\Zero_g}$.  By
definition, for a commutative ring $R$ the group $\Sp_{2g}(R)$ consists of all $2g \times 2g$ matrices
$M$ with entries in $R$ such that $M^t \Omega_g M = \Omega_g$.  Define $\SpLie_{2g}(R)$
to be the additive group of $2g \times 2g$ matrices $A$ with entries in $R$ such that
$A^t \Omega_g + \Omega_g A = 0$.

We then have the following theorem.

\begin{theorem}
\label{theorem:sppabel}
Fix $g,p \geq 3$ such that $p$ is odd.  Then
$$\HH_1(\Sp_{2g}(\Z,p);\Z) \cong \SpLie_{2g}(\Z/p) \quad \text{and} \quad 
[\Sp_{2g}(\Z,p),\Sp_{2g}(\Z,p)] = \Sp_{2g}(\Z,p^2).$$
\end{theorem}

\begin{remarks}
\mbox{}
\begin{enumerate}
\item Theorem \ref{theorem:sppabel} was proved independently by Perron \cite{PerronModPAbel}, Putman \cite{PutmanModPAbel},
and Sato \cite{SatoModPAbel}.  All three papers were inspired by a basic result of Lee and Szczarba \cite{LeeSzczarba}
which calculates the abelianizations of congruence subgroups of $\SL_n(\Z)$.  See Exercise \ref{exercise:leeszczarba}
for a discussion of Lee and Szczarba's work.
\item For $p$ even, Sato \cite{SatoModPAbel} proved that there is a short exact sequence
$$0 \longrightarrow \HHH_2 \longrightarrow \HH_1(\Sp_{2g}(\Z,p);\Z) \longrightarrow \SpLie_{2g}(\Z/p) \longrightarrow 0.$$
Here $\HHH_2 = \HH_1(\Sigma_{g,1};\Z/2)$.
\end{enumerate}
\end{remarks}

\ParagraphHeading{The map from $\Sp_{2g}(\Z,p)$ to $\SpLie_{2g}(\Z/p)$.}
To prove Theorem \ref{theorem:sppabel}, we first construct a surjective homomorphism
$$\psi : \Sp_{2g}(\Z,p) \longrightarrow \SpLie_{2g}(\Z/p).$$
Consider $M \in \Sp_{2g}(\Z,p)$.  By definition, we can write $M = \One_{2g} + p \cdot A$ for some matrix
$A$.  Define $\psi(M) = A$ modulo $p$.  We have $\psi(M) \in \SpLie_{2g}(\Z/p)$; indeed, by
definition we have $\Omega_g = M^t \Omega_g M$, so
\begin{align*}
\Omega_g &= (\One_{2g} + p \cdot A)^t \Omega_g (\One_{2g} + p \cdot A) \\
&= \Omega_g + p \cdot (A^t \Omega_g + \Omega_g A) + p^2 A^{t} \Omega_g A.
\end{align*}
This implies that modulo $p$, we have $A^t \Omega_g + \Omega_g A = 0$.
Next, we prove that $\psi$ is a homomorphism.  Consider $M,N \in \Sp_{2g}(\Z,p)$ with
$M = \One_{2g} + p A$ and $N = \One_{2g} + p B$.  We then have
$M N = \One_{2g} + p(A+B) + p^2 A B$, so modulo $p$ we have $\psi(MN) = A + B$.

It remains to show that $\psi$ is surjective.  If $p$ is prime, then this is easy.  Indeed, $\psi$
is equivariant with respect to the conjugation actions of $\Sp_{2g}(\Z)$ on $\Sp_{2g}(\Z,p)$ and
$\SpLie_{2g}(\Z/p)$.  The latter factors through $\Sp_{2g}(\Z/p)$, so the image of $\psi$ is
a $\Sp_{2g}(\Z/p)$-subrepresentation of $\SpLie_{2g}(\Z/p)$.  However,
it is well-known that $\SpLie_{2g}(\Z/p)$ is an irreducible $\Sp_{2g}(\Z/p)$-representation,
so since $\psi$ is not trivial it must be surjective (we remark that this uses the fact
that $p$ is odd -- the abelian group $\SpLie_{2g}(\Z/2)$ is {\em not} an irreducible representation
of $\Sp_{2g}(\Z/2)$).  We leave the general case as an exercise.

\begin{exercise}
Prove that $\psi$ is surjective by constructing elements of $\Sp_{2g}(\Z,p)$ that map
to generators of $\SpLie_{2g}(\Z/p)$.
\end{exercise}

\ParagraphHeading{Finishing the proof of Theorem \ref{theorem:sppabel}.}
Observe now that $\Ker(\psi) = \Sp_{2g}(\Z,p^2)$.  Since the target of $\psi$ is abelian, this implies
that
$$[\Sp_{2g}(\Z,p),\Sp_{2g}(\Z,p)] \subset \Sp_{2g}(\Z,p^2).$$
Theorem \ref{theorem:sppabel} thus follows from the following lemma.

\begin{lemma}
\label{lemma:expressascommutators}
If $g,p \geq 3$ and $p$ is odd, then $\Sp_{2g}(\Z,p^2) \subset [\Sp_{2g}(\Z,p),\Sp_{2g}(\Z,p)]$.
\end{lemma}

\noindent
The proof of Lemma \ref{lemma:expressascommutators} follows from a direct matrix calculation.  Rather
than giving the details (which can be found in \cite[Lemma 2.4]{PutmanModPAbel} or \cite[Proposition 10.1]{SatoModPAbel}), 
we will give an exercise which outlines a proof of a somewhat easier result for
$\SL_n(\Z)$.  The proof of Lemma \ref{lemma:expressascommutators} follows the same basic pattern, though
the details are more complicated.

\begin{exercise}
\label{exercise:leeszczarba}
Let $\SL_n(\Z,p)$ be the kernel of the natural map $\SL_n(\Z) \rightarrow \SL_n(\Z/p)$.  In this
exercise, you will prove a theorem of Lee and Szczarba \cite{LeeSzczarba} which says that for
$n \geq 3$, we have $\HH_1(\SL_n(\Z,p);\Z) \cong \SLLie_n(\Z/p)$.  Here $\SLLie_n(\Z/p)$ is
the abelian group of $n \times n$ matrices over $\Z/p$ with trace $0$.  We remark that in this
exercise, we do not need to assume that $p$ is odd.
\begin{enumerate}
\item Construct a surjective homomorphism $\SL_n(\Z,p) \rightarrow \SLLie_n(\Z/p)$
whose kernel is $\SL_n(\Z/p^2)$.  Conclude that
$[\SL_n(\Z,p),\SL_n(\Z,p)] \subset \SL_n(\Z,p^2)$.
\item Prove that $\SL_n(\Z/p^2) \subset [\SL_n(\Z,p),\SL_n(\Z,p)]$.  You will need the following
theorem of Bass, Milnor, and Serre \cite{BassMilnorSerre}.  For $1 \leq i,j \leq n$ such that
$i \neq j$, let $e_{ij}$ be the {\em elementary matrix} which is obtained from the $n \times n$
identity matrix by changing the entry at position $(i,j)$ to $1$.  Bass, Milnor, and Serre
proved that for $n \geq 3$ and $q \geq 2$, the group $\SL_n(\Z,q)$ is normally generated
(as a subgroup of $\SL_n(\Z)$) by the set
$\Set{$e_{ij}^q$}{$1 \leq i,j \leq n$, $i \neq j$}$.
\end{enumerate}
\end{exercise}

\section*{Lecture 4 : The second rational homology group of $\Mod_g(p)$}
\addtocounter{section}{1}
\setcounter{theorem}{0}
\setcounter{figure}{0}

In this final lecture, we turn to the second homology group of $\Mod_g(p)$.  To 
set the stage, let us first recall what happens for the first homology group.  Powell \cite{Powell} 
proved that $\HH_1(\Mod_g;\Z) = 0$ for $g \geq 3$.  See
\cite[\S 5.1]{FarbMargalitPrimer} for an easier proof (due to Harer).  Hain \cite{HainTorelli}
later proved that $\HH_1(\Mod_g(p);\Q)=0$ for $g \geq 3$.  In other words,
over $\Q$ the first homology group of $\Mod_g$ does not change when you pass
to the finite-index subgroup $\Mod_g(p)$.  Our proof
of Theorem \ref{theorem:modpabel} in Lecture 3 is essentially an elaboration of
Hain's proof.

Harer \cite{HarerH2} proved that $\HH_2(\Mod_g;\Z) \cong \Z$ for $g \geq 4$.
The following theorem says that a similar result holds 
for $\Mod_g(p)$ as long as we work over $\Q$.

\begin{theorem}[{Putman, \cite{PutmanSecondHomologyLevel}}]
\label{theorem:h2mod}
For $g \geq 5$ and $p \geq 2$, we have $\HH_2(\Mod_g(p);\Q) \cong \Q$.
\end{theorem}

\noindent
The first part of this lecture is devoted to motivating why one
might expect such a theorem to hold.  Next, we show
that Theorem \ref{theorem:h2mod} is equivalent to a weak
form of homological stability for $\Mod_g(p)$.  Finally, we
discuss how to prove this stability result.

\ParagraphHeading{The transfer map.}
What happens to the homology of a group when you pass to a finite-index
subgroup?  The following standard lemma says that over $\Q$, the homology
can get larger but it cannot get smaller.

\begin{lemma}
\label{lemma:transfer}
If $G$ is a finite-index subgroup of $\Gamma$, then the map
$\HH_k(G;\Q) \rightarrow \HH_k(\Gamma;\Q)$ is surjective for all
$k \geq 1$.
\end{lemma}

\begin{example}
If $\Gamma$ is a free abelian group and $G$ is a finite-index subgroup of $\Gamma$, then
there exists some $\ell \geq 1$ such that $\ell \cdot \Gamma < G < \Gamma$.  Consequently, we have
$G \cong \Gamma$ and it is easy to see that the map $\HH_k(G;\Q) \rightarrow \HH_k(\Gamma;\Q)$ is
an isomorphism for all $k \geq 1$.
\end{example}

\begin{example}
If $\Gamma$ is a free group of rank $n$ and $G$ is a proper finite-index
subgroup of $\Gamma$, then $G$ is a free group of rank strictly greater than $n$.  Consequently,
the rank of $\HH_1(G;\Q)$ is strictly greater than the rank of $\HH_1(\Gamma;\Q) \cong \Q^n$.
\end{example}

\begin{remark}
Lemma \ref{lemma:transfer} is false over $\Z$.  Indeed, if $\Gamma$ is a finite group, then
$G = 1$ is a finite-index subgroup of $\Gamma$.  The homology groups of $G$ are all trivial,
but $\Gamma$ can certainly have nontrivial homology groups.  However, Lemma \ref{lemma:transfer}
does imply that the homology groups of $\Gamma$ are all torsion.
\end{remark}

\begin{proof}[{Proof of Lemma \ref{lemma:transfer}}]
Let $n = [\Gamma:G]$.  The key to the proof is the {\em transfer map}.  Letting $R$ be
a commutative ring and letting $i_{\ast} : \HH_k(G;R) \rightarrow \HH_k(\Gamma;R)$ 
be the map induced by the inclusion $i : G \hookrightarrow \Gamma$, the transfer
map is a homomorphism $\tau_k : \HH_k(\Gamma;R) \rightarrow \HH_k(G;R)$ satisfying
\begin{equation}
\label{eqn:transfer}
i_{\ast}(\tau_k(x)) = n x \quad \text{for all $x \in \HH_k(\Gamma;R)$}.
\end{equation}
The existence of $\tau_k$ immediately implies the lemma.  Indeed, if $R = \Q$, then
$\frac{1}{n} \tau_k$ is a right-inverse to $i_{\ast}$, so $i_{\ast}$ is surjective.

The transfer map is {\em not} induced by a homomorphism $\Gamma \rightarrow G$; rather,
it is constructed at the level of chains.  The construction goes as follows.  Let
$B\Gamma$ be a $K(\Gamma,1)$ and let $\rho : BG \rightarrow B\Gamma$ be the cover
corresponding to $G < \Gamma$.  Thus $BG$ is a $K(G,1)$ and $\rho$ is a degree $n$
cover.  Define a map $\hat{\tau}_k : C_k(B\Gamma;R) \rightarrow C_k(BG;R)$ as follows.
If $\sigma^k \in C_k(B\Gamma;R)$ is a singular $k$-simplex, then there are exactly
$n$ singular $k$-simplices $\tilde{\sigma}^k_1, \ldots, \tilde{\sigma}^k_n$ in $BG$
satisfying $\rho_{\ast}(\tilde{\sigma}^k_i) = \sigma^k$.  Define
$$\hat{\tau}_k(\sigma^k) = \tilde{\sigma}^k_1 + \cdots + \tilde{\sigma}^k_n.$$ 

\begin{exercise}
The maps $\hat{\tau}_k$ commute with the boundary operators, so we have a commutative
diagram
$$\begin{CD}
\cdots @>{\partial}>> C_k(\Gamma;R) @>{\partial}>> C_{k-1}(\Gamma;R) @>{\partial}>> \cdots\\
@.                    @VV{\hat{\tau}_k}V           @VV{\hat{\tau}_{k-1}}V             @.\\
\cdots @>{\partial}>> C_k(G;R) @>{\partial}>> C_{k-1}(G;R) @>{\partial}>> \cdots
\end{CD}$$
\end{exercise}

\noindent
By this exercise, the maps $\hat{\tau}_k$ induce maps $\tau_k : \HH_k(\Gamma;R) \rightarrow \HH_k(G;R)$
in homology.  Clearly the maps $\tau_k$ satisfy \eqref{eqn:transfer}, so we are done.
\end{proof}

\ParagraphHeading{Borel stability.}
Lemma \ref{lemma:transfer} implies that $\HH_2(\Mod_g(p);\Q) \cong \Q^m$ for some $m \geq 1$.  Why
should we expect that $m=1$?  There is a very fruitful analogy between the mapping class group
and lattices in Lie groups.  The Borel stability theorem says that for the classical sequences
of arithmetic lattices, passing to finite-index subgroups does not change their rational homology
in a stable range.  For example, denote $\SL_n(\Z)$ by $\Gamma_n$ and define
$$\Gamma_n(p) = \Ker(\SL_n(\Z) \rightarrow \SL_n(\Z/p))$$
for $p \geq 2$.  Here the map $\SL_n(\Z) \rightarrow \SL_n(\Z/p)$ comes from reducing all
the entries in the matrices modulo $p$.  Borel's theorem then takes the following form.

\begin{theorem}[{Borel, \cite{BorelStability}}]
\label{theorem:borelstab}
For $k \geq 1$, there exists some $N_k \geq 1$ such that if $n \geq N_k$ and $p \geq 2$, then
$\HH_k(\Gamma_n(p);\Q) \cong \HH_k(\Gamma_n;\Q)$.
\end{theorem}

\noindent
One can view Theorem \ref{theorem:h2mod} as an analogue (for $k=2$) of Theorem \ref{theorem:borelstab} for
the mapping class group.  As we mentioned, Hain \cite{HainTorelli} proved a similar theorem
for $k=1$.  It would be very interesting to extend this to $k \geq 3$.

\ParagraphHeading{Reduction to stability.}
We now discuss the proof of Theorem \ref{theorem:h2mod}.  In \cite{HarerStability}, Harer
proved that the homology groups of $\Mod_g$ satisfy homological stability.  See \cite{IvanovStability}
or \cite{WahlStability} for more readable proofs of this.  One special
case of Harer's theorem is as follows.  Choose a subsurface $S$ of $\Sigma_g$ such that
$S \cong \Sigma_{g-1,1}$ (see Figure \ref{figure:stability}).  Extending mapping classes
on $S$ by the identity to $\Sigma_g$, we get a homomorphism $\Mod_{g-1,1} \rightarrow \Mod_g$.  Harer
proved that for each $k$, there exists some $N_k$ such that if $g \geq N_k$, then the
induced map $\HH_k(\Mod_{g-1,1};\Z) \rightarrow \HH_k(\Mod_g;\Z)$ is an isomorphism.

\Figure{figure:stability}{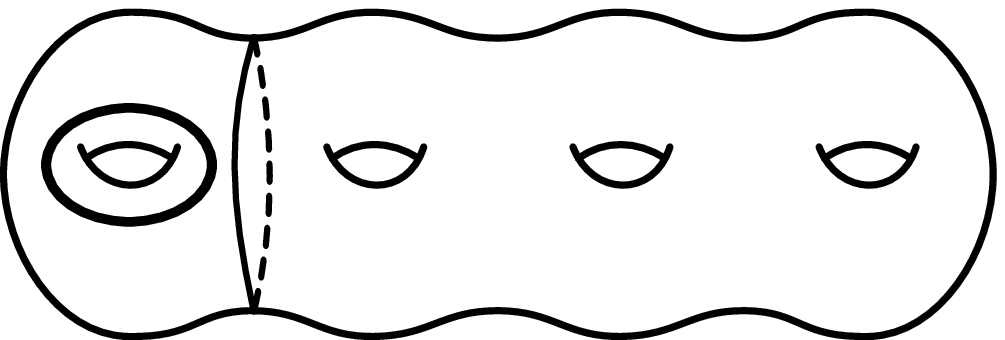}{$S$ is a subsurface of $\Sigma_g$ satisfying $S \cong \Sigma_{g-1,1}$
and $\gamma$ is a simple closed nonseparating curve contained in $\Sigma_g \setminus S$}

The following result asserts that something similar happens for $\Mod_{g}(p)$ for $k=2$.

\begin{lemma}
\label{lemma:stability1}
For $g \geq 5$ and $p \geq 2$, let $i : \Mod_{g-1,1}(p) \rightarrow \Mod_g(p)$ be the restriction
of the map $\Mod_{g-1,1} \rightarrow \Mod_g$ described above.
Then the induced map $i_{\ast} : \HH_2(\Mod_{g-1,1}(p);\Q) \rightarrow \HH_2(\Mod_g(p);\Q)$
is surjective.
\end{lemma}

\noindent
Observe that Theorem \ref{theorem:h2mod} implies Lemma \ref{lemma:stability1}.  Indeed, if
Theorem \ref{theorem:h2mod} is true, then we have a commutative diagram of the form
$$\begin{CD}
\HH_2(\Mod_{g-1,1}(p);\Q) @>>> \HH_2(\Mod_g(p);\Q)\\
@VVV                           @VV{\cong}V \\
\HH_2(\Mod_{g-1,1};\Q)    @>>> \HH_2(\Mod_g;\Q)
\end{CD}$$
By Lemma \ref{lemma:transfer}, the map $\HH_2(\Mod_{g-1,1}(p);\Q) \rightarrow \HH_2(\Mod_{g-1,1};\Q)$
is a surjection, and Harer's stability theorem say that the map 
$\HH_2(\Mod_{g-1,1};\Q) \rightarrow \HH_2(\Mod_g;\Q)$ is an isomorphism as long as $g \geq 5$.  We
conclude that the map $\HH_2(\Mod_{g-1,1}(p);\Q) \rightarrow \HH_2(\Mod_g(p);\Q)$ is a surjection,
as desired.

Somewhat surprisingly, Lemma \ref{lemma:stability1} also implies Theorem \ref{theorem:h2mod}.

\begin{lemma}
\label{lemma:stabilityimpliesiso}
Lemma \ref{lemma:stability1} implies Theorem \ref{theorem:h2mod}.
\end{lemma}

\noindent
Before proving Lemma \ref{lemma:stabilityimpliesiso}, we need
some more abstract nonsense.  Let $G$ be a finite-index normal subgroup of $\Gamma$.
The conjugation action of $\Gamma$ on $G$ induces an action of $\Gamma$ on $\HH_k(G;\Q)$.  Recall
that if $M$ is a vector space upon which $\Gamma$ acts, then the {\em coinvariants} of the action, 
denoted $M_{\Gamma}$, is the quotient $M/I$ with $I$ the subspace spanned by
$\Set{$x-g(x)$}{$x \in M$, $g \in \Gamma$}$.  In other words, $M_{\Gamma}$ is the largest
quotient of $M$ upon which $\Gamma$ acts trivially.  Since $\Gamma$ acts trivially on
its own homology groups $\HH_k(\Gamma;\Q)$, the map $\HH_k(G;\Q) \rightarrow \HH_k(\Gamma;\Q)$
factors through $(\HH_k(G;\Q))_{\Gamma}$.  We then have the following exercise.

\begin{exercise}
\label{exercise:supertransfer}
$\HH_k(\Gamma;\Q) \cong (\HH_k(G;\Q))_{\Gamma}$.  Hint : carefully study the proof
of Lemma \ref{lemma:transfer}.  Alternatively, this can be proven using the Hochschild-Serre
spectral sequence of the extension
$$1 \longrightarrow G \longrightarrow \Gamma \longrightarrow \Gamma/G \longrightarrow 1.$$
\end{exercise}

\begin{proof}[{Proof of Lemma \ref{lemma:stabilityimpliesiso}}]
Assume that Lemma \ref{lemma:stability1} holds.  Since $\HH_2(\Mod_g;\Q) \cong \Q$, to
prove that $\HH_2(\Mod_g(p);\Q) \cong \Q$, it is enough to prove that
$\HH_2(\Mod_g(p);\Q) \cong \HH_2(\Mod_g;\Q)$.  By Exercise \ref{exercise:supertransfer},
this is equivalent to showing that the action of $\Mod_g$ on $\HH_2(\Mod_g(p);\Q)$ is
trivial.  This can be checked on a generating set.  The group $\Mod_g$ is generated by the set
of Dehn twists about nonseparating simple closed curves.  Consider such a Dehn twist $T_{\gamma}$.
As is shown in Figure \ref{figure:stability}, we can find a subsurface $S$ of $\Sigma_g$ such
that $S \cong \Sigma_{g-1,1}$ and such that $\gamma \subset \Sigma_g \setminus S$.  Let
$i : \Mod_{g-1,1}(p) \rightarrow \Mod_g(p)$ be the map induced by the subsurface inclusion
$S \hookrightarrow \Sigma_g$.  Clearly $T_{\gamma}$ commutes with $\Image(i)$, so $T_{\gamma}$
acts trivially on $i_{\ast}(\HH_2(\Mod_{g-1,1}(p);\Q))$.  Lemma \ref{lemma:stability1} implies
that $i_{\ast}(\HH_2(\Mod_{g-1,1}(p);\Q)) = \HH_2(\Mod_g(p);\Q)$, so we conclude that
$T_{\gamma}$ acts trivially on $\HH_2(\Mod_g(p);\Q)$, as desired.
\end{proof}

\ParagraphHeading{A weaker result suffices.}
In fact, we do not need the full strength of Lemma \ref{lemma:stability1} to prove
Theorem \ref{theorem:h2mod}.  If $\gamma$ is the isotopy class of a nonseparating simple
closed curve on $\Sigma_g$, then denote by $(\Mod_g(p))_{\gamma}$ the subgroup of
$\Mod_g(p)$ which stabilizes $\gamma$.  Examining the proof of Lemma \ref{lemma:stabilityimpliesiso},
it is clear that the following lemma also implies Theorem \ref{theorem:h2mod}.

\begin{lemma}
\label{lemma:stability2}
For $g \geq 5$ and $p \geq 2$, let $\gamma$ be a nonseparating simple closed curve on
$\Sigma_{g}$.  Then the map $\HH_2((\Mod_g(p))_{\gamma};\Q) \rightarrow \HH_2(\Mod_g(p);\Q)$
is surjective.
\end{lemma}

\ParagraphHeading{The homological stability machine.}
Homological stability theorems are known for many different sequences of groups, and
there is now a standard procedure for proving them.  This procedure goes back to Quillen.
See \cite{HatcherWahl} for an excellent discussion of how this machine applies
to many different sequences of groups and \cite{WahlStability} for a detailed
exposition of how to apply it to the mapping class group.  The input for
this machine consists of a sequence of highly connected spaces upon which
the groups in question act.  For $\Mod_g$, Harer proved homological stability
by studying the action on the following space.

\begin{definition}
The {\em nonseparating curve complex}, denoted $\CNosep_g$, is the simplicial
complex whose $k$-simplices are sets $\{\gamma_0,\ldots,\gamma_k\}$ of isotopy classes
of simple closed curves on $\Sigma_g$ that can be realized such that the $\gamma_i$
are all disjoint and $\Sigma_g \setminus (\gamma_0 \cup \cdots \cup \gamma_k)$
is connected.
\end{definition}

One's first impulse, of course, is to apply this machine to $\Mod_g(p)$ to try
to prove Lemma \ref{lemma:stability2}.  Unfortunately, it does not quite work.  The
problem is that if it did work, then it would prove that $\HH_k(\Mod_g(p);\Z)$ is stable,
and this is false even for $k=1$ (see Theorem \ref{theorem:modpabel}).  The machine
does, however, give the following result.

\begin{lemma}
\label{lemma:weakstability}
For $g \geq 5$ and $p \geq 2$, the map
$$\bigoplus_{\gamma \in (\CNosep_g)^{(0)}} \HH_2((\Mod_g(p))_{\gamma};\Q) \longrightarrow \HH_2(\Mod_g(p);\Q)$$
is surjective.
\end{lemma}

\ParagraphHeading{Rescuing the machine.}
To deduce Lemma \ref{lemma:stability2} from Lemma \ref{lemma:weakstability}, we need to prove
that if $g \geq 5$ and $p \geq 2$, then for any two nonseparating simple closed curves $\gamma$ and
$\gamma'$ on $\Sigma_g$, the images of $\HH_2((\Mod_g(p))_{\gamma};\Q)$ and
$\HH_2((\Mod_g(p))_{\gamma'};\Q)$ in $\HH_2(\Mod_g(p);\Q)$ are the same.  A standard
result about surfaces called the ``change of coordinates principle'' (see \cite{FarbMargalitPrimer})
says that there exists some $\phi \in \Mod_g$ such that $\phi(\gamma) = \gamma'$.  This implies
that
$$\phi (\Mod_g(p))_{\gamma} \phi^{-1} = (\Mod_g(p))_{\gamma'}.$$
Thus the action of $\Mod_g$ on $\HH_2(\Mod_g(p);\Q)$ takes the image
of $\HH_2((\Mod_g(p))_{\gamma};\Q)$ to the image of $\HH_2((\Mod_g(p))_{\gamma'};\Q)$.  We conclude
that Lemma \ref{lemma:weakstability} would following from the following lemma.

\begin{lemma}
\label{lemma:trivialaction}
For $g \geq 5$ and $p \geq 2$, if $\gamma$ is a nonseparating simple closed curve on $\Sigma_g$,
then $\Mod_g$ acts trivially on the image of $\HH_2((\Mod_g(p))_{\gamma};\Q)$ in
$\HH_2(\Mod_g(p);\Q)$.
\end{lemma}

The proof of Lemma \ref{lemma:trivialaction} is quite complicated (see \cite{PutmanSecondHomologyLevel}).
Rather than give it here, we will prove an analogous result about $\Gamma_n = \SL_n(\Z)$.  Recall
that
$$\Gamma_n(p) = \Ker(\SL_n(\Z) \longrightarrow \SL_n(\Z/p)).$$
There is a natural map $\Gamma_{n-1} \rightarrow \Gamma_{n}$ which takes a matrix $A \in \Gamma_{n-1}$
to the matrix $\MatTwoTwo{1}{0}{0}{A} \in \Gamma_{n}$.  In \cite{CharneyCongruence},
Charney proves the following theorem.

\begin{theorem}[{Charney, \cite[Proposition 5.5]{CharneyCongruence}}]
\label{theorem:charney}
For $n \geq 2$ and $k,p \geq 1$, the group $\Gamma_n$ acts trivially on the image
of $\HH_k(\Gamma_{n-1}(p);\Q)$ in $\HH_k(\Gamma_n(p);\Q)$.
\end{theorem}

\noindent
Though she does not say it in her paper, it is not hard to combine this with the
standard homological stability machinery to give a new proof of the Borel stability
theorem for $\SL_n(\Z)$.  After the proof, we will comment on the relation between
the proofs of Theorem \ref{theorem:charney} and Lemma \ref{lemma:trivialaction}.

\begin{proof}[{Proof of Theorem \ref{theorem:charney}}]
For $1 \leq i,j \leq n$ such that
$i \neq j$, let $e_{ij}$ be the {\em elementary matrix} which is obtained from the $n \times n$
identity matrix by changing the entry at position $(i,j)$ to $1$.
The group $\Gamma_n$ is generated by the set $\Set{$e_{ij}$}{$1 \leq i,j \leq n$, $i \neq j$}$.
However, we do not need this entire set.

\begin{exercise}
For $2 \leq i,j \leq n$ such that $i \neq j$, prove that $[e_{i1},e_{1j}] = e_{ij}$.
\end{exercise}

\noindent
Consequently, we only need to prove that $e_{1j}$ and $e_{i1}$ act trivially
for $2 \leq i,j \leq n$.  We will give the details for $e_{1j}$.  The other case is similar.

Consider the subgroups
$$G = \Set{$\left(
\begin{array}{c|c}
1      & c_2 \ \cdots\  c_n \\ \hline
0      & \multirow{3}{*}{\raisebox{-5mm}{\scalebox{2}{$A$}}} \\
\raisebox{2mm}{\vdots} & \\
0      & \\
\end{array}\right)$}
{$A \in \Gamma_{n-1}(p)$, $c_2,\dots,c_n \in p\Z$}$$
and
$$\hat{G} = \Set{$\left(
\begin{array}{c|c}
1      & c_2 \ \cdots\  c_n \\ \hline
0      & \multirow{3}{*}{\raisebox{-5mm}{\scalebox{2}{$A$}}} \\
\raisebox{2mm}{\vdots} & \\
0      & \\
\end{array}\right)$}
{$A \in \Gamma_{n-1}(p)$, $c_2,\dots,c_n \in \Z$}$$
of $\Gamma_n$.  The conjugation action of $e_{1j}$ on $\Gamma_n$ takes $G$ and $\hat{G}$ to
themselves.  We have $\Gamma_{n-1}(p) \subset G \subset \Gamma_n(p)$, so it
is enough to show that $e_{1j}$ acts trivially on $\HH_k(G;\Q)$.  Since
$e_{1j} \in \hat{G}$, the action of $e_{1j}$ on $\HH_k(\hat{G};\Q)$ is trivial.
It is therefore enough to show that the inclusion $G \hookrightarrow \hat{G}$ induces
an isomorphism $\HH_k(G;\Q) \cong \HH_k(\hat{G};\Q)$.

Define
$$K = \Set{$\left(
\begin{array}{c|c}
1      & c_2 \ \cdots\  c_n \\ \hline
0      & \multirow{3}{*}{\raisebox{-5mm}{\scalebox{2}{$1$}}} \\
\raisebox{2mm}{\vdots} & \\
0      & \\
\end{array}\right)$}
{$c_2,\dots,c_n \in p\Z$} \subset G.$$
Observe that $K \cong \Z^{n-1}$.  We have a short exact sequence
$$\begin{CD}
1 @>>> K @>>> G @>>> \Gamma_{n-1}(p) @>>> 1.
\end{CD}$$
Similarly, setting
$$\hat{K} = \Set{$\left(
\begin{array}{c|c}
1      & c_2 \ \cdots\  c_n \\ \hline
0      & \multirow{3}{*}{\raisebox{-5mm}{\scalebox{2}{$1$}}} \\
\raisebox{2mm}{\vdots} & \\
0      & \\
\end{array}\right)$}
{$c_2,\dots,c_n \in \Z$} \subset \hat{G}$$
we have a short exact sequence
$$\begin{CD}
1 @>>> \hat{K} @>>> \hat{G} @>>> \Gamma_{n-1}(p) @>>> 1.
\end{CD}$$
These short exact sequences fit into a commutative diagram of the form
$$\begin{CD}
1 @>>> K    @>>> G @>>> \Gamma_{n-1}(p) @>>> 1 \\
@.     @VVV      @VVV   @VV{=}V          @.\\
1 @>>> \hat{K} @>>> \hat{G} @>>> \Gamma_{n-1}(p) @>>> 1.
\end{CD}$$
By Lemma \ref{lemma:grouphomology5} (a sort of $5$-lemma for group homology), it 
is enough to show that the inclusion map $K \hookrightarrow \hat{K}$ induces
an isomorphism $\HH_k(K;\Q) \cong \HH_k(\hat{K};\Q)$.  However, we have
$\hat{K} \cong \Z^{n-1}$ and $K = p \cdot \hat{K}$, so this is immediate.  We
remark that this final step is the only place we use the fact that we are
working over $\Q$.
\end{proof}

\begin{lemma}[{$5$-lemma for group homology}]
\label{lemma:grouphomology5}
Fix a commutative ring $R$.  Consider a commutative diagram
$$\begin{CD}
1 @>>> A_1    @>>> B_1 @>>> C @>>> 1 \\
@.     @VVV      @VVV   @VV{=}V          @.\\
1 @>>> A_2 @>>> B_2 @>>> C @>>> 1
\end{CD}$$
of short exact sequences of groups.  Assume that
the map $A_1 \rightarrow A_2$ induces an isomorphism
$\HH_k(A_1;R) \cong \HH_k(A_2;R)$ for all $k$.
Then the map $B_1 \rightarrow B_2$ induces an
isomorphism $\HH_k(B_1;R) \cong \HH_k(B_2;R)$ for all $k$.
\end{lemma}
\begin{proof}
Associated each of our short exact sequences is a Hochschild-Serre spectral
sequence in group homology.  The assumptions in the lemma imply that the
induced map between these spectral sequences is an isomorphism on the $E^2$-page.  This
implies that it converges an an isomorphism on the $E^{\infty}$-pages, so we
obtain isomorphisms $\HH_k(B_1;R) \cong \HH_k(B_2;R)$ for all $k$.
\end{proof}

\ParagraphHeading{Relation between Theorem \ref{theorem:charney} and Lemma \ref{lemma:trivialaction}.}
The proof of Lemma \ref{lemma:trivialaction} in \cite{PutmanSecondHomologyLevel} is
not nearly as short as the proof above of Theorem \ref{theorem:charney}.  However, they
share some features.  The key to the proof of Theorem \ref{theorem:charney} above
is the short exact sequence
$$1 \longrightarrow K \longrightarrow G \longrightarrow \Gamma_{n-1}(p) \longrightarrow 1.$$
A similar role in the proof of Lemma \ref{lemma:trivialaction} is played by
various analogues for $\Mod_g(p)$ of the Birman exact sequence which was discussed in Lecture 1.  However,
the kernels of these Birman exact sequences are not free abelian, so passing to finite-index
subgroups actually changes their isomorphism types.  This necessitates making a number of rather intricate
twisted group cohomology computations.

\ParagraphHeading{Recap.}
The above proof that $\HH_2(\Mod_g(p);\Q) \cong \Q$ was a little involved, so let us recap the main
steps.
\begin{enumerate}
\item We first showed that $\HH_2(\Mod_g(p);\Q) \cong \Q$ if and only if $\HH_2(\Mod_g(p);\Q)$
satisfies a weak form of homological stability (see Lemma \ref{lemma:stabilityimpliesiso}).
\item We then attempted to use the standard homological stability machinery to 
prove that $\HH_2(\Mod_g(p);\Q)$ stabilizes.  This failed, but it showed that
$\HH_2(\Mod_g(p);\Q)$ is ``concentrated'' on stabilizers of simple closed curves
(see Lemma \ref{lemma:weakstability}).
\item The proof is reduced to showing that all stabilizers of simple closed
curves give the same ``chunk'' of $\HH_2(\Mod_g(p);\Q)$.
\end{enumerate}
It seems reasonable to conjecture that $\HH_k(\Mod_g(p);\Q) \cong \HH_k(\Mod_g;\Q)$ in a stable range
for all $k$.  The first two steps above can
be easily generalized to $\HH_k(\Mod_g(p);\Q)$ for $k \geq 3$.  The real
difficulty is the third step.  It would be very interesting to extend it
to the higher homology groups.

\noindent
{\raggedright
Andrew Putman\\
Department of Mathematics\\
Rice University, MS 136 \\
6100 Main St.\\
Houston, TX 77005\\
E-mail: {\tt andyp@rice.edu}}

\end{document}